\documentclass{mcom-l}

\usepackage{epsf}
\usepackage{geometry}
\usepackage{listings} 
\usepackage{algorithmic,algorithm}
\usepackage{multirow}
\usepackage{enumerate}
\usepackage{booktabs}

\usepackage{accents}
\usepackage[pdftex]{graphicx}
\usepackage{amscd}
\usepackage[pdftex]{color} 
\usepackage[pdftex,colorlinks]{hyperref}
\usepackage{lscape}
\usepackage{indentfirst}
\usepackage{latexsym}
\usepackage{amsmath, amsfonts, amssymb,mathrsfs}
\usepackage{verbatim}
\usepackage{bm}

\usepackage{amsmath}
\usepackage{todonotes} 
\usepackage[extra]{tipa}

\newcommand{\vertiii}[1]{{\left\vert\kern-0.25ex\left\vert\kern-0.25ex\left\vert #1 
    \right\vert\kern-0.25ex\right\vert\kern-0.25ex\right\vert}}

\hypersetup{
    bookmarks=true,         
    unicode=false,          
    pdftoolbar=true,        
    pdfmenubar=true,        
    pdffitwindow=true,      
    pdftitle={},    
    pdfauthor={Johannes Kraus},     
    pdfsubject={},   
    pdfnewwindow=true,      
    pdfkeywords={}, 
    colorlinks=true,       
    linkcolor=red,          
    citecolor=blue,        
    filecolor=magenta,      
    urlcolor=cyan           
}

\everymath{\displaystyle}

\newtheorem{assumption}{Assumption}

\newtheorem{theorem}{Theorem}[section]
\newtheorem{lemma}[theorem]{Lemma}

\theoremstyle{definition}
\newtheorem{definition}[theorem]{Definition}
\newtheorem{example}[theorem]{Example}

\theoremstyle{remark}
\newtheorem{remark}[theorem]{Remark}

\numberwithin{equation}{section}

\def\w{{{w}}}

\def\r{{{r}}}

\def\bV{{\boldsymbol V}}
\def\bu{{\boldsymbol u}}
\def\bv{{\boldsymbol v}}
\renewcommand{\div}{\operatorname{div}}

\renewcommand{\div}{\operatorname{div}}

\def\bH{\boldsymbol H}

\def\bV{\boldsymbol V}

\def\bQ{\boldsymbol Q}

\def\bbf{\boldsymbol f}
\def\bf{\boldsymbol f}

\def\bu{\boldsymbol u}
\def\bv{\boldsymbol v}
\def\bbu{\bar{\boldsymbol u}}
\def\bbv{\bar{\boldsymbol v}}
\def\bw{\boldsymbol w}

\def\bp{\boldsymbol p}
\def\bq{\boldsymbol q}

\def\bz{\boldsymbol z}

\def\Divv{\text{Div}}

\begin{document}

\title[New stability analysis of saddle-point problems]
{A new practical framework for the stability analysis of perturbed saddle-point problems and applications}



\author{Qingguo Hong*}
\address{Department of Mathematics, Pennsylvania State University, University Park, PA 16802, USA}
\curraddr{}
\email{huq11@psu.edu}
\thanks{}

\author{Johannes Kraus}
\address{Faculty of Mathematics, University of Duisburg-Essen, Thea-Leymann-Straße 9, Essen 45127, Germany}
\curraddr{}
\email{johannes.kraus@uni-due.de}
\thanks{J. Kraus and M. Lymbery acknowledge the support of this work by the Deutsche Forschungsgemeinschaft (DFG, German Research Foundation) as part of the project ``Physics-oriented solvers for multicompartmental poromechanics'' under grant number 456235063}

\author{Maria Lymbery}
\address{Faculty of Mathematics, University of Duisburg-Essen, Thea-Leymann-Straße 9, Essen 45127, Germany}
\curraddr{}
\email{maria.lymbery@uni-due.de}
\thanks{}

\author{Fadi Philo}
\address{Faculty of Mathematics, University of Duisburg-Essen, Thea-Leymann-Straße 9, Essen 45127, Germany}
\curraddr{}
\email{fadi.philo@uni-due.de}
\thanks{}


\date{}

\dedicatory{}

\begin{abstract}
In this paper we prove a new abstract stability result for perturbed saddle-point problems based on a norm fitting technique.
We derive the stability condition according to Babu\v{s}ka's theory from a small inf-sup condition, similar to  the famous
Ladyzhenskaya-Babu\v{s}ka-Brezzi (LBB) condition, and the other standard assumptions in Brezzi's theory, in a combined
abstract norm. The construction suggests to form the latter from individual {\it fitted} norms that are composed from
proper seminorms. 

This abstract framework not only allows for simpler (shorter) proofs of many stability results but also
guides the design of parameter-robust norm-equivalent preconditioners. These benefits are demonstrated on 
mixed variational formulations of generalized Poisson, Stokes, vector Laplace and Biot's equations. 
\end{abstract}

\maketitle

\section{Introduction}\label{sec:introduction} 
Saddle-point problems (SPPs) arise in various areas of computational science and engineering ranging
from computational fluid dynamics~\cite{GiraultRaviart1979finite,Temam1984NavierStokes,Glowinski2003finite}, elasticity~\cite{ArnoldWinther2002mixed,ErnGuermond2004theory,Braess2007finite}, 
and electromagnetics~\cite{Monk2003finite,Boffi2013mixed} to
computational finance~\cite{KwokZheng2018saddlepoint}. 
Moreover, SPPs play a vital role
in the context of image reconstruction~\cite{Hall1979computer},
model order reduction~\cite{Schilders2008model},
constrained optimization~\cite{Gill1981practical},
optimal control~\cite{Battermann1998preconditioners}, and
parameter identification~\cite{Burger2002iterative}, to mention only a few but important applications.

In the mathematical modeling of multiphysics phenomena described by (initial-) boundary-value problems
for systems of partial differential equations, SPPs often naturally arise and are frequently posed in a
variational formulation. Mixed finite element methods and other discretization techniques can be and have been
successfully used for their discretization and numerical solution, see,
e.g.~\cite{Brezzi1991mixed,ElmanSilvesterWathen2005finite,BenziGolubLiesen2005numerical,Boffi2013mixed}
and the references therein. 

The pioneering works laying the foundations of the solution theory for SPPs 
have been conducted by Jind\v{r}ich Ne\v{c}as, Olga Ladyzhenskaya, Ivo Babu\v{s}ka,
and Franco Brezzi~\cite{Necas1967methodes,Ladyzhenskaya1969mathematical,Babuska1971error,Brezzi1974}, 
see also the contributions~\cite{BabuskaAziz1972survey,LadyzhenskayaSolonnikov1976problems}.

Designing and analyzing discretizations and solvers for SPPs require a careful study of the mapping properties
of the underlying operators. Of particular interest are their continuity and stability, which not only guarantee
the well-posedness of (continuous and discrete) mathematical models but also provide the basis for error
estimates and a convergence analysis of iterative methods and preconditioners, see,
e.g.~\cite{LoghinWathen2004analysis,ElmanSilvesterWathen2005finite,mardal2011preconditioning,Boffi2013mixed},
for a review see also~\cite{BenziGolubLiesen2005numerical}.

Saddle-point problems/systems are of a two-by-two block form and characterized by an operator/matrix
of the form
\begin{equation}\label{A_classic_block_form}
\mathcal{A} = \begin{pmatrix}
A & B_1^T \\
B_2 & -C
\end{pmatrix}
\end{equation}
where $A$ and $C$ denote positive semidefinite operators/matrices and $B_1^T$ the adjoint/transpose of an 
operator/matrix $B_1$. We consider the symmetric case in this paper where $A^T=A$, $C^T=C$, and
$B_1=B_2=B$. Problems in which $C \neq 0$ are often referred to as {\it perturbed saddle-point problems}. 

In~\cite{Zulehner2011nonstandard} a technique has been proposed to determine norms for parameter-dependent 
SPPs providing necessary and sufficient conditions 
for their well-posedness and leading to robust estimates of the solution in terms of the data. 
A drawback of this approach, however, 
is that these conditions are often hard to verify in practice 
as the operators inducing the norms are 
defined only implicitly.

More general SPPs in which $A$ (and $C$) are allowed to be nonsymmetric and $B_1\neq B_2$
have also been studied by many authors, see, e.g.,~\cite{Nicolaides1982existence,Brezzi1991mixed,Boffi2013mixed},
and the references therein.
Their analysis, in general is more complicated and is mostly done following the {\it monolithic} approach, i.e., imposing
conditions on $\mathcal{A}$ rather than on $A$, $B_1$, $B_2$, and $C$ separately. 

Our work is motivated by the stability analysis of variational problems occurring in
poromechanics~(cf.~\cite{Coussy2004poromechanics}), a subarea of continuum mechanics
which originates from the early works of~Terzaghi and Biot~\cite{Terzaghi1925erdbaumechanic,Biot1941general}.
Various formulations of Biot's consolidation model have been considered and analyzed since it had been introduced
in~\cite{Biot1941general,Biot1955theory}, including two-field (\cite{MuradLoula1992improved,MuradLoula1994stability}),
three-field (\cite{PhillipsWheeler2007coupling1,PhillipsWheeler2007coupling2,oyarzua2016locking,lee2017parameter,HongKraus2018,chen2020robust}),
and four-field-formulations (\cite{Yi2014convergence,Lee2016robust,Kumar2020conservative}),
for generalizations to several fluid networks as considered in multiple network poroelastic theory (MPET),
see also~\cite{Bai_etal1993multi,Guo2019on,Hong2019conservativeMPET,Hong2020ParameterUzawa,HongEtAl2020parameter,lee2018mixed,TullyVentikos2011cerebral}.

Although they typically relate more than two physical fields, or quantities
of interest (except for the two-field formulation of Biot's model), 
the variational problems arising from the above-mentioned formulations--subject to a proper grouping or rather
aggregation of variables--result in symmetric two-by-two block systems of saddle point form characterized by a self-adjoint
operator~$\mathcal{A}$.

The abstract framework presented in the next section of this paper applies to such saddle-point operators. After introducing
some notation, we recall the classical stability results of Babu\v{s}ka and Brezzi for classical (unperturbed) SPPs.
Next, we focus on perturbed (symmetric) SPPs, initially summarizing some of the additional conditions which, together
with the Ladyzhenskaya-Babu\v{s}ka-Brezzi (LBB) condition (small inf-sup condition), imply the necessary and sufficient
stability condition of Babu\v{s}ka (big inf-sup condition). Our main theoretical result then follows in Section~\ref{sec:stability}
where we propose a generalization of the classical Brezzi conditions for the analysis of perturbed SPPs with $C\neq 0$.
These new conditions imply the Babu\v{s}ka condition. The {fitted norms on which they are} based provide a constructive tool for
designing norm-equivalent preconditioners.

This paper does not discuss discretizations and discrete variants of inf-sup conditions. However, the proposed framework
directly translates to discrete settings where it also allows for shorter and simpler proofs of the well-posedness of
discrete models and error-estimates for stable discretizations.
 
\section{Abstract framework}\label{sec:abstract}
\subsection{Notation and problem formulation}
Consider two Hilbert spaces $V$ and $Q$ equipped with the norms 
$\Vert \cdot \Vert_V$ and $\Vert \cdot \Vert_Q$ induced by the scalar products 
$(\cdot,\cdot)_V$ and $(\cdot,\cdot)_Q$, respectively. 
We denote their product space by $Y:=V\times Q$ and endow it  
with the norm $\Vert \cdot \Vert_Y$ defined by 
\begin{equation}\label{y_norm}
\Vert y \Vert_Y^2 = (y,y)_Y = (v,v)_V + (q,q)_Q = \Vert v\Vert_V^2 + \Vert q\Vert_Q^2  \qquad \forall y = 
(v;q):= \begin{pmatrix}v\\ q\end{pmatrix}\in Y. 
\end{equation}

Next, we introduce an abstract 
bilinear form $\mathcal{A}((\cdot \,;\cdot),(\cdot \,;\cdot))$ 
on $Y\times Y$ defined by
\begin{equation}\label{bilinear_form}
\mathcal{A}((u;p),(v;q)):=a(u,v)+b(v,p)+b(u,q)-c(p,q)
\end{equation}
for some symmetric positive semidefinite (SPSD)
bilinear forms $a(\cdot,\cdot)$ on 
$V\times V$, $c(\cdot,\cdot)$ on $Q \times Q$, i.e., 
\begin{align}
a(u,v) & = a(v,u) \qquad \forall u,v\in V, \\
a(v,v) & \ge 0 \qquad \forall v\in V, \\
c(p,q) & = c(q,p) \qquad \forall p,q\in Q, \\
c(q,q) & \ge 0 \qquad \forall q\in Q.
\end{align}
and a bilinear form 
$b(\cdot,\cdot)$ on $V \times Q$.

We assume that $a(\cdot,\cdot)$, $b(\cdot,\cdot)$ and $c(\cdot,\cdot)$ are 
continuous with respect to the norms $\Vert \cdot\Vert_V$ and $\Vert \cdot \Vert_Q$, i.e.,
\begin{align}
a(u,v) & \le \bar{C}_a \Vert u\Vert_V \Vert v\Vert_V \qquad \forall u,v\in V, 
\label{a_cont}\\
b(v,q) & \le \bar{C}_b \Vert v\Vert_V \Vert q\Vert_Q  \qquad \forall v\in V,\forall q\in Q, 
\label{b_cont}\\
c(p,q) & \le \bar{C}_c \Vert p\Vert_Q \Vert q\Vert_Q  \qquad \forall p,q\in Q. 
\label{c_cont}
\end{align} 

Then each of these bilinear forms defines a bounded linear operator as follows: 
\begin{subequations}
\begin{align}
A: V\rightarrow V': \,\, & \langle Au,v\rangle_{V'\times V} = a(u,v), \qquad \forall u,v\in V, \\
C: Q\rightarrow Q': \,\, & \langle Cp,q\rangle_{Q'\times Q} = c(p,q), \qquad \forall p,q\in Q, \\
B: V\rightarrow Q': \,\, & \langle Bv,q\rangle_{Q'\times Q}  = b(v,q), \qquad \forall v \in V, \forall q\in Q, \\ 
B^T: Q\rightarrow V': \,\, & \langle v, B^T q\rangle_{V\times V'} = b(v,q), \qquad \forall v \in V, \forall q\in Q.
\end{align}
\end{subequations}
Here, $V'$ and $Q'$ denote the dual spaces of $V$ and $Q$ and 
$\langle \cdot, \cdot \rangle $ the corresponding duality pairings.  

Associated with the bilinear form $\mathcal{A}$ defined in~\eqref{bilinear_form} we consider the following abstract perturbed saddle-point problem 
\begin{equation}\label{PSPP}
\mathcal{A}((u;p),(v;q)) = \mathcal{F}((v;q))\qquad \forall v\in V, \forall q \in Q
\end{equation}
which can also be written as 
$$\mathcal{A}(x,y) = \mathcal{F}(y)\qquad \forall y \in Y,$$
thereby using the definitions $x=(u;p)$ and $y=(v;q)$, or, in operator form 
\begin{equation}\label{PSPP_o}
\mathcal{A} x = \mathcal{F}
\end{equation}
where 
\begin{align}
\mathcal{A}: Y\rightarrow Y': \,\, & \langle \mathcal{A}x,y
\rangle_{Y'\times Y} = \mathcal{A}(x,y), \qquad \forall x,y\in Y
\end{align}
and $\mathcal{F}\in Y'$, i.e., $\mathcal{F}: Y\rightarrow \mathbb{R}: \mathcal{F}(y)=\langle \mathcal{F},y \rangle_{Y'\times Y}$ 
for all $y \in Y.$

The operator $\mathcal{A}$ can also be represented in block form by 
\begin{equation}\label{A_block_form}
\mathcal{A} = \begin{pmatrix}
A & B^T \\
B & -C
\end{pmatrix}.
\end{equation}

Problem~\eqref{PSPP} (and \eqref{PSPP_o}) is called a perturbed saddle-point problem (in operator form) 
when $c(\cdot,\cdot) \not \equiv 0$ and a classical saddle-point problem in 
the case $c(\cdot,\cdot)\equiv 0$. 
\subsection{Babu\v{s}ka's and Brezzi's conditions for stability of saddle-point problems}

As is well-known from~\cite{Babuska1971error}, the abstract variational problem~\eqref{PSPP} is well-posed
under the following necessary and sufficient conditions (i) and (ii) given in the following theorem. 

\begin{theorem}[Babu\v{s}ka \cite{Babuska1971error}]\label{thm:1} 
Let $\mathcal{F}\in Y'$ be a bounded linear functional. 
Then the saddle-point problem~\eqref{PSPP} is well-posed if and only if there exist positive 
constants $\bar{C}$ and $\alpha$
for which 
the conditions 
\begin{equation}\label{boundedness_A}
\mathcal{A}(x, y) \le \bar{C} \Vert x\Vert_{Y} \Vert y\Vert_{Y} \qquad \forall x, y \in Y,
\end{equation}
\begin{equation}\label{inf_sup_A}
\inf_{x \in Y} \sup_{y \in Y} \frac{\mathcal{A}(x,y)}{\Vert x\Vert_{Y} \Vert y\Vert_{Y}} \ge \underline{\alpha}>0
\end{equation}
hold.
The solution $x$ then 
satisfies the stability estimate 
\begin{equation*}
\Vert x\Vert_{Y} \le \frac{1}{ \underline{\alpha}} \sup_{y \in Y}\frac{\mathcal{F}(y)}{\Vert y\Vert_{Y}}
=: \frac{1}{ \underline{\alpha}}\Vert \mathcal{F}\Vert_{Y'}.
\end{equation*}
\end{theorem}

\begin{remark}
Estimate~\eqref{boundedness_A} ensures continuity, that is, boundedness of 
the operator $\mathcal{A}$ from above, 
whereas \eqref{inf_sup_A} is a stability condition, sometimes 
referred to as Babu\v{s}ka condition, which grants boundedness of $\mathcal{A}$ from below.
\end{remark}
{
Using the operator notations introduced in \eqref{A_block_form}, the conditions \eqref{boundedness_A} and \eqref{inf_sup_A} can be rewritten as 
\begin{equation}\label{cor_bounded_A}
 \underline{\alpha} \Vert y\Vert_{Y} \le \Vert \mathcal{A} y\Vert_{Y'}\le  \bar{C} \Vert y\Vert_Y~~ \hbox{for all}~~y\in Y.
\end{equation}
In~\cite{Zulehner2011nonstandard}, the condition \eqref{cor_bounded_A} is characterized by two equivalent conditions as stated in the following theorem.
\begin{theorem}[Zulehner \cite{Zulehner2011nonstandard}]
 If there are constants $\underline{\gamma}_{v}, \bar{\gamma}_{v},\underline{\gamma}_{q}, \bar{\gamma}_{q}>0$ such that
 \begin{equation}\label{Z1}
 \underline{\gamma}_{v}\Vert v\Vert_{V}^{2} \leq a(v, v)+ \left(\sup_{q\in Q}\frac{b(v,q)}{\Vert q\Vert_Q} \right)^2 \leq \bar{\gamma}_{v}\Vert v\Vert_{V}^{2} \quad\forall  v \in V,
 \end{equation}
and
 \begin{equation}\label{Z2}
 \underline{\gamma}_{q}\Vert q\Vert_{Q}^{2} \leq c(q, q)+  \left(\sup_{v\in V}\frac{b(v,q)}{\Vert v\Vert_V} \right)^2 \leq \bar{\gamma}_{q}\Vert q\Vert_{Q}^{2} \quad \forall q \in Q,
 \end{equation}
then \eqref{cor_bounded_A}
is satisfied with constants $\underline{\alpha}, \bar{C}>0$ that depend only on $\underline{\gamma}_{v}, \bar{\gamma}_{v},\underline{\gamma}_{q}, \bar{\gamma}_{q}$.
And, vice versa, if the estimates \eqref{cor_bounded_A} are satisfied with constants $\underline{\alpha}, \bar{C}>0$, then the estimates \eqref{Z1} and \eqref{Z2} are satisfied with constants $\underline{\gamma}_{v}, \bar{\gamma}_{v},\underline{\gamma}_{q}, \bar{\gamma}_{q}>0$ that depend only on $\underline{\alpha}, \bar{C}>0$.
\end{theorem}
\begin{remark}
The two conditions \eqref{boundedness_A} and \eqref{inf_sup_A}, entangling the bilinear forms $a(\cdot,\cdot), b(\cdot,\cdot)$ and $c(\cdot,\cdot)$, are equivalent to the two conditions \eqref{Z1} and \eqref{Z2} which entangle the bilinear forms $a(\cdot,\cdot)$ and $b(\cdot,\cdot)$ and the bilinear forms $c(\cdot,\cdot)$ and $b(\cdot,\cdot)$, respectively.
However, verifying the two big conditions \eqref{Z1} and \eqref{Z2} is sometimes difficult or even impractical. Our aim is to propose a framework untangling the bilinear forms $a(\cdot,\cdot), b(\cdot,\cdot)$ and $c(\cdot,\cdot)$ and impose Brezzi-type conditions, in particular a small coercivity condition on $a(\cdot,\cdot)$ and a small inf-sup condition on $b(\cdot,\cdot)$. 
\end{remark}
}


For the classical saddle-point problem, i.e., $c(\cdot,\cdot)\equiv 0$, the following theorem which we formulate 
under the conditions that $a(\cdot,\cdot)$ is symmetric positive semidefinite 
and 
\begin{equation}\label{b_empty}
{\rm Ker}(B^T):= \{ q\in Q: b(v,q)=0 \,\, \forall v\in V \} = \emptyset
\end{equation}
has been proven in~\cite{Brezzi1974}, see 
also \cite{Brezzi1991mixed,Boffi2013mixed}. 

\begin{theorem}[Brezzi \cite{Brezzi1974}]\label{brezzi1}
Assume that the bilinear forms $a(\cdot,\cdot)$ and $b(\cdot,\cdot)$ are continuous on $V\times V$ and on 
$V\times Q$, respectively, $a(\cdot,\cdot)$ is symmetric positive semidefinite, and also that 
{
\begin{equation}\label{a_coerc_ker}
a(v,v)\ge \underline{C}_a \Vert v\Vert_V^2 \qquad\forall v \in {\rm Ker}(B),
\end{equation}
}
\begin{equation}\label{small_inf_sup}
\inf_{q \in Q}\sup_{v\in V} \frac{ b(v,q)}{\Vert v\Vert_V \Vert q\Vert_Q} 
\ge \beta >0, 
\end{equation}
hold. Then the classical saddle-point problem (problem~\eqref{PSPP} with $c(\cdot,\cdot)\equiv 0$) is well-posed.

\end{theorem}

\begin{remark}
Note that if ${\rm Ker}(B^T)\neq \emptyset$, the statement of the above theorem (Theorem~\ref{brezzi1}) remains valid if 
we identify any two elements $q_1$, $q_2$ for which $q_0:=q_1-q_2$ is an element of ${\rm{Ker}}(B^T)$, i.e., 
replacing the space $Q$ with the quotient space $Q/{\rm{Ker}}(B^T)$ and also the norm  
$\Vert \cdot \Vert_Q$ with $\Vert \cdot \Vert_{Q/{\rm{Ker}}(B^T)}$, the latter being defined by 
$$
\Vert q \Vert_{Q/{\rm{Ker}}(B^T)} = \inf_{q_0\in {\rm Ker}(B^T)} \Vert q+q_0\Vert_Q. 
$$
In this case, the solution $p$ is only unique up to an arbitrary element $p_0 \in {\rm Ker}(B^T)$.
\end{remark}

For the classical saddle-point problem Brezzi's stability condition~\eqref{small_inf_sup} 
and the continuity of $a(\cdot,\cdot)$ imply Babu\v{s}ka's stability condition~\eqref{inf_sup_A}, 
see~\cite{Demkowicz2006BabuskaBrezzi}, where it has also been shown that from~\eqref{inf_sup_A}
it follows \eqref{small_inf_sup} and {the inf-sup condition for $a(\cdot,\cdot)$ in the kernel of $B$}, the latter 
being equivalent to the coercivity estimate~\eqref{a_coerc_ker} if $a(\cdot,\cdot)$ is symmetric positive semidefinite.   

Obviously, the stability condition~\eqref{inf_sup_A} directly applies to perturbed saddle-point problems, 
a reason why they can be studied using Babu\v{s}ka's theory. 
However, conditions~\eqref{a_coerc_ker} and \eqref{small_inf_sup} together with the continuity of $a(\cdot,\cdot)$, $b(\cdot,\cdot)$ 
and $c(\cdot,\cdot)$ in general are not sufficient to guarantee the stability condition~\eqref{inf_sup_A} 
when $c(\cdot,\cdot)\not \equiv 0$.
Additional conditions to ensure~\eqref{inf_sup_A} have been studied, for example in~\cite{Brezzi1991mixed,Braess1996stability,Boffi2013mixed}.

In~\cite{Brezzi1991mixed} it has been shown that a condition on the kernel of $B^T$ can be used as an additional assumption 
to ensure well-posedness of the perturbed saddle-point problem, that is, in particular, 
for Babu\v{s}ka's inf-sup condition~\eqref{inf_sup_A} to hold. 
This condition is expressed in terms of the following auxiliary problem 
\begin{equation}\label{1.54}
\epsilon(p_0,q)_Q+c(p_0,q) = - c(p^{\perp},q), \qquad \forall q\in {\rm Ker}(B^T)
\end{equation}
and requires the following general assumption: 

\begin{assumption}\label{assumption0}
There exists a $\gamma_0>0$ such that for every $p^{\perp}\in ({\rm Ker}(B^T))^{\perp}$ 
and for every $\epsilon >0$ it holds that {the norm of the solution $p_0\in {\rm Ker}(B^T)$ of~\eqref{1.54} is bounded by 
$ \Vert p_0\Vert_Q \le \frac{1}{\gamma_0} \Vert p^{\perp}\Vert_Q$}.  
\end{assumption}
 
The theorem then reads as follows:

\begin{theorem}[Brezzi and Fortin \cite{Brezzi1991mixed}]\label{brezzi2}
Assume that $a(\cdot, \cdot)$, $b(\cdot,\cdot)$ and $c(\cdot,\cdot)$ are 
continuous bilinear forms on $V\times V$, on $V\times Q$, and on 
$Q\times Q$, respectively. Further assume that $a(\cdot,\cdot)$ and 
$c(\cdot,\cdot)$ are symmetric positive semidefinite. Finally, 
let~\eqref{a_coerc_ker}, \eqref{small_inf_sup} (conditions (i) and (ii) 
from Theorem~\ref{brezzi1}) and Assumption~\ref{assumption0} be satisfied. Then for every $f\in V'$ and every $g\in {\rm Im} (B)$ problem~\eqref{PSPP} 
with $\mathcal{A}$ as defined in~\eqref{bilinear_form} and  $\mathcal{F}(y) = \langle f,v\rangle_{V'\times V}+ \langle g,q\rangle_{Q' \times Q}$ 
has a unique solution 
$x=(u;p)$ in $Y=V\times Q/M$ where 
\begin{equation*}
M = {\rm Ker} (B^T) \cap {\rm Ker}(C).
\end{equation*} 
Moreover, the estimate 
\begin{equation*}
\Vert u\Vert_V +\Vert p\Vert_{Q/{\rm Ker}(B^T)} \le K (\Vert f\Vert_{V'} +\Vert g\Vert_{Q'})
\end{equation*}
holds with a constant $K$ only depending on $\bar{C}_a$, $\bar{C}_c$, $\underline{C}_a$, 
$\beta$ and $\gamma_0$.
\end{theorem}

\begin{remark}
The result in~\cite{Brezzi1991mixed} is more general than Theorem~\ref{brezzi2} in that it applies also to non-symmetric but positive semidefinite $a(\cdot,\cdot)$. We are 
considering only the case of symmetric positive semidefinite $a(\cdot,\cdot)$ in this paper. 
\end{remark}

In order to ensure the boundedness (continuity) of the symmetric positive semidefinite bilinear form $c(\cdot,\cdot)$ 
with respect to the norm $\Vert \cdot\Vert_Q$, and more generally the boundedness of $\mathcal{A}(\cdot,\cdot)$ 
with respect to the combined norm $\Vert \cdot \Vert_Y = (\Vert \cdot \Vert_V^2 +\Vert \cdot \Vert_Q^2)^{1/2}$, it is 
natural to include the contribution of $c(\cdot,\cdot)$ in the norm $\Vert \cdot \Vert_Q$, e.g., by 
defining $\Vert \cdot \Vert_Q$ via 
\begin{equation}\label{q_splitting}
\Vert q\Vert^2_Q = \vert q\vert_Q^2 + t^2 c(q,q), \qquad \forall q\in Q,
\end{equation}
for a proper seminorm or norm $\vert \cdot \vert_Q$ and a parameter $t\in [0,1]$.

{As it has been shown} in~\cite{Braess1996stability} the stability of the perturbed saddle-point problem 
then can be proven under Brezzi's conditions for the classical saddle-point problem and the additional condition  
\begin{equation}\label{braess_condition}
\inf_{u\in V} \sup_{(v;q)\in V\times Q} \frac{a(u,v)+b(u,q)}{\vertiii{(v;q)}} \ge \gamma>0
\end{equation} 
where $\vertiii{\cdot}$ is defined by
\begin{equation}\label{triple_norm}
\vertiii{(v;q)}^2 := \Vert v\Vert_V^2 + \vert q\vert_Q^2 + t^2 c(q,q), \qquad t\in [0,1],
\end{equation}
and provides a specific choice of $\Vert \cdot \Vert_Y$, i.e., $\Vert \cdot\Vert_Y = \vertiii{\cdot}$.
The corresponding theorem then reads as: 
\begin{theorem}[Braess \cite{Braess1996stability}]\label{braess_theorem}
Assume that the classical saddle-point problem {with $a(\cdot,\cdot)$ being SPSD and $c(\cdot,\cdot)\equiv 0$} 
is stable, i.e., Brezzi's conditions~\eqref{a_coerc_ker} and 
\eqref{small_inf_sup} are fulfilled. If in addition condition~\eqref{braess_condition} holds with $\gamma >0$
and we choose $t>0$ in~\eqref{triple_norm} for $c(\cdot,\cdot)\not \equiv 0$, 
then the perturbed saddle-point problem~\eqref{PSPP} is stable under the norm $\Vert \cdot\Vert_Y:=\vertiii{\cdot}$ 
and the stability constant $ \underline{\alpha}$ in~\eqref{inf_sup_A} depends only on $\beta$, $\underline{C}_a$ and $\gamma$ 
and the choice of $t$. 
\end{theorem}

\begin{remark}
Note that as it can easily be seen if $a(\cdot,\cdot)$ is symmetric positive semidefinite, 
condition~\eqref{braess_condition} is equivalent to the condition that there exists a constant 
$\gamma'>0$ such that
\begin{equation}\label{braess_condition2}
\frac{a(u,u)}{\Vert u\Vert_V} + \sup_{q\in Q} \frac{b(u,q)}{\vert q\vert_Q + t c(q,q)}\ge \gamma' \Vert u \Vert_V \qquad 
\forall u\in V.
\end{equation}
Moreover, as shown in~\cite{Braess1996stability}, then~\eqref{braess_condition2} is also equivalent to 
the condition that there exists a constant $\gamma^{\prime \prime}>0$ such that
\begin{equation}\label{braess_condition3}
\sup_{(v;q)\in Y} \frac{\mathcal{A}((u;0),(v;q))} {\vertiii{(v;q)}} \ge \gamma^{\prime \prime} \Vert u \Vert_V \qquad \forall u\in V.
\end{equation} 
\end{remark}

Since~\eqref{braess_condition} is an inf-sup condtion for $(a(\cdot,\cdot)+b(\cdot,\cdot))$ which can be interpreted as a big 
inf-sup condition on $\mathcal{A}$ for $p=0$ under the specific norm $\Vert \cdot \Vert_Y=\vertiii{\cdot}$, 
see~\eqref{braess_condition3}, Theorem~\ref{braess_theorem} still does not provide us with the desired stability result 
in terms of conditions on $a(\cdot,\cdot)$, $b(\cdot,\cdot)$ and $c(\cdot,\cdot)$ separately. On the other hand, 
Theorem~\ref{brezzi2} requires the solution of the auxiliary problem~\eqref{1.54} on ${\rm Ker}(B^T)$ 
for which one has to verify Assumption~\ref{assumption0} 
which, in some situations, is a difficult task. 

Our aim is to avoid the latter and still impose Brezzi-type conditions, 
in particular a small inf-sup condition on $b(\cdot,\cdot)$. In the next section, we will prove a theorem (Theorem \ref{our_theorem}) which ensures the stability of the perturbed saddle-point problem (12) under conditions which are equivalent to the conditions in Brezzi’s theorem (Theorem \ref{brezzi1}) when the perturbation term vanishes. 
Moreover, our approach provides a framework suited for finding norms in which stability can 
be shown and allows for simplifying and shortening proofs based on the result of Babu\v{s}ka.  

\subsection{A new framework for the stability analysis of perturbed saddle-point problems}\label{sec:stability}

The key idea for studying and verifying the stability of perturbed saddle-point problems we follow in this 
paper is to construct proper norms as part of an abstract framework which applies to a variational formulation of 
various multiphysics models. As we have already observed in the previous subsection, a norm-splitting 
of the form~\eqref{q_splitting} is quite natural if the symmetric positive semidefinite perturbation form 
$c(\cdot,\cdot)$ is not identical to zero. For fixed $t>0$, the norm defined in~\eqref{q_splitting} is 
equivalent to the norm defined by 
\begin{equation} \label{Q_norm}
\Vert q\Vert_Q^2 :=\vert q\vert_Q^2+c(q,q) =: \langle \bar{Q}q,q \rangle_{Q'\times Q}.
\end{equation}
Note that the assumption that $\Vert \cdot \Vert_Q$ is a full norm induced by an inner product 
under which $Q$ is a Hilbert space implies that the seminorm $\vert \cdot \vert_Q$ corresponds to an SPSD 
bilinear form $d(\cdot,\cdot):Q\times Q \rightarrow \mathbb{R}$, i.e., $\vert q\vert_Q^2=d(q,q)$. Consequently, the form 
$c(p,q)+d(p,q)$ is symmetric positive definite (SPD) and defines a linear operator $\bar{Q}: Q\rightarrow Q'$ by $\langle \bar{Q}p,q\rangle : = c(p,q)+d(p,q)$.

Now we introduce the following splitting of the norm $\Vert \cdot \Vert_V$ defined by
\begin{align}\label{V_norm}
\Vert v \Vert^2_{V} &:= \vert v\vert^2_{V} +\vert v \vert^2_{b}
\end{align}
where $\vert \cdot \vert_V$ is a proper seminorm, which is a norm on ${\rm Ker}(B)$ satisfying  
$$
\vert v\vert_V^2 \eqsim  a(v,v) \qquad \forall v \in {\rm Ker}(B)
$$ 
and $\vert \cdot\vert_b$ is defined by 
\begin{align}\label{seminorm_b}
\vert v \vert^2_{b}:= \langle Bv, \bar{Q}^{-1}Bv \rangle_{Q' \times Q} = 
\Vert Bv \Vert_{Q'}^2.  
\end{align}
Here, $\bar{Q}^{-1}:Q'\rightarrow Q$ is an isometric isomorphism (Riesz isomorphism) since $\bar{Q}$ is an 
isometric isomorphism, i.e.,   
\begin{align*}
\Vert Bv\Vert_{Q'}^2& = \Vert \bar{Q}^{-1}Bv\Vert_Q^2 = (\bar{Q}^{-1}Bv,\bar{Q}^{-1}Bv)_Q 
= \langle \bar{Q}\bar{Q}^{-1}Bv,\bar{Q}^{-1}Bv \rangle_{Q'\times Q} 
= \langle Bv, \bar{Q}^{-1}Bv\rangle_{Q'\times Q}.
\end{align*}

\begin{remark}
Note that both $\vert \cdot \vert_V$ and $\vert \cdot \vert_b$ can be seminorms as long as they add up to 
a full norm. 
Likewise, only the sum of the seminorms $\vert \cdot\vert_Q$ and $c(\cdot,\cdot)$ has to define a norm. 
In some particular situations, it is also useful to identify certain of the involved seminorms with 
$0$, in which case the corresponding splitting becomes a trivial splitting. 
The splitting~\eqref{V_norm} is closely related to a Schur complement type operator, corresponding to the 
modified (regularized) bilinear form resulting from $\mathcal{A}((\cdot;\cdot),(\cdot;\cdot))$ by replacing 
$c(\cdot,\cdot)$ with $(\cdot,\cdot)_Q$. 
\end{remark}

In order to present our main theoretical result, we give the following definition. 

\begin{definition}\label{A1}
{
For two Hilbert spaces $V$ and $Q$, a norm $\Vert \cdot\Vert_V$ on $V$ and a norm $\Vert \cdot\Vert_Q$ on $Q$ 
are called fitted}
if they satisfy the splittings~\eqref{Q_norm} and~\eqref{V_norm}, respectively, where $\vert \cdot\vert_Q$ 
is a seminorm on~$Q$ and $\vert \cdot \vert_V$ and $\vert \cdot\vert_b$ are seminorms on $V,$ 
the latter defined according to~\eqref{seminorm_b}. 
\end{definition}

\begin{remark}\label{remark_alternative}
{
Note that the norm fitting can also be performed by first fixing the full norm on $V$ (instead of the full norm on $Q$ as described above). 
Exploiting the structure of the problem, in the latter case one uses the following norm splittings
\begin{equation}\nonumber
\begin{array}{l}
\Vert v \Vert_V : =  \vert v \vert_V^2 + a(v,v)=:\langle \bar{V}v,v \rangle_{V'\times V},\\
\Vert q \Vert_Q : =  \vert q \vert_Q^2 + \vert q\vert_b^2,
\end{array}
\end{equation}
where $\bar{V}:V\rightarrow V'$ is a linear operator, $\vert q\vert_Q^2$ is equivalent to $c(q,q)$ and $\vert q\vert_b^2=:\langle B^T q, \bar{V}^{-1}B^Tq \rangle_{V'\times V}$.}
\end{remark}
\begin{theorem}\label{our_theorem}
Let $\Vert \cdot\Vert_V$ and $\Vert \cdot \Vert_Q$ be {fitted} norms according to Definition~\ref{A1}, 
which immediately implies the continuity of $b(\cdot,\cdot)$ and $c(\cdot,\cdot)$
in these norms with $\bar{C}_b=1$ and
$\bar{C}_c=1$, cf.~\eqref{b_cont}--\eqref{c_cont}.
Consider the bilinear form $\mathcal{A}((\cdot;\cdot),(\cdot;\cdot))$ 
defined in~\eqref{bilinear_form} where $a(\cdot,\cdot)$ is continuous,
i.e., \eqref{a_cont} holds, and $a(\cdot,\cdot)$ and 
$c(\cdot,\cdot)$ are symmetric positive semidefinite. 
Assume, further, that $a(\cdot,\cdot)$ satisfies the coercivity estimate
\begin{equation}\label{coerc_a}
a(v,v)\ge \underline{C}_a \vert v\vert_V^2, \qquad \forall v\in V,
\end{equation}
and that there exists a constant 
$\underline{\beta}>0$ such that 
\begin{equation}\label{inf_sup_b}
\sup_{\substack{v\in V\\ v\neq 0}} 
\frac{b(v,q)}{\Vert v\Vert_V}\ge \underline{\beta}  \vert q\vert_Q, \qquad \forall q\in Q.
\end{equation} 
Then the bilinear form $\mathcal{A}((\cdot;\cdot),(\cdot;\cdot))$ is continuous and inf-sup stable under the combined norm 
$\Vert \cdot \Vert_Y$ defined in~\eqref{y_norm}, i.e., the conditions~\eqref{boundedness_A} and \eqref{inf_sup_A} hold.
\end{theorem}
{
Before presenting the proof of Theorem~\ref{our_theorem}, we show an auxiliary result and make some remarks. 
\begin{lemma} \label{equi:inf-sup}
The inf-sup condition: 
there exists a constant 
$\underline{\beta}>0$ such that 
\begin{equation}\label{inf_sup_b0}
  \sup_{\substack{v\in V\\ v\neq 0}} \frac{b(v,q)}{\Vert v\Vert_V}\ge \underline{\beta}  \vert q\vert_Q, \qquad \forall q\in Q,
\end{equation} 
is equivalent to the condition: for any $q \in Q$, there exists $v \in V$, such that
\begin{equation}\label{infsup:equi}
b(v, q)=\vert q\vert_Q^{2}\quad\text{and}\quad\Vert v\Vert_V \leq   \underline{\beta}^{-1}\vert q\vert_Q.
\end{equation}
\end{lemma}
\begin{proof}
Obviously, \eqref{infsup:equi} implies \eqref{inf_sup_b0}. Hence, it remains to prove that \eqref{inf_sup_b0} implies \eqref{infsup:equi}. 
Let $K_Q=\{q\in Q: \vert q\vert_Q= 0 \}$ and define $\displaystyle \Vert q\Vert_{Q/K_Q}:=\vert q\vert_Q$ which is a norm on the quotient space $Q/K_Q$.
For any $q \in Q$, where the class in $Q/K_Q$ which $q$ belongs to is also denoted by $q$, there exists $f \in (Q/K_Q)^{\prime}$ s.t. $f(q)=\Vert q\Vert_{Q/K_Q}^{2}$ and $\Vert f\Vert_{(Q/K_Q)'}=\Vert q\Vert_{Q/K_Q}$. Since $B$ is onto, we can find $v$ s.t. $B v=f$ and by the open mapping theorem, we can find $v$ with $\Vert v\Vert_V \leq  \underline{\beta}^{-1}\Vert f\Vert_{(Q/K_Q)'} =\underline{\beta}^{-1}\Vert q\Vert_{Q/K_Q}=\underline{\beta}^{-1}\vert q\vert_Q$ and $b(v, q)=\langle B v, q\rangle=f(q)=\Vert q\Vert_{Q/K_Q}^{2}=\vert q\vert_Q^2$.
\end{proof}}
\begin{remark}
The continuity of $b(\cdot,\cdot)$ readily follows from
$$
b(v,q)=\langle B v, q \rangle_{Q' \times Q} = \langle \bar{Q} \bar{Q}^{-1} B v, q \rangle_{Q' \times Q}
=(\bar{Q}^{-1} B v, q)_Q \le \Vert \bar{Q}^{-1} B v \Vert_Q \Vert q \Vert_Q \le \Vert v \Vert_V \Vert q \Vert_Q.
$$
If $\vert \cdot \vert_V$ is induced by the bilinear form $a(\cdot,\cdot)$ then the continuity of $a(\cdot,\cdot)$
also follows directly from the definition of the fitted norms.
\end{remark}
\begin{remark}
Theorem~\ref{our_theorem} is a generalization of Theorem~\ref{brezzi1} in the sense that given two norms
$\Vert \cdot \Vert_{Q,{\rm eqv}}$ and $\Vert \cdot \Vert_{V,{\rm eqv}}$ under which the conditions of Theorem~\ref{brezzi1}
are satisfied, one can always find two fitted equivalent norms $\Vert \cdot \Vert_{Q} \eqsim \Vert \cdot \Vert_{Q,{\rm eqv}}$
and $\Vert \cdot \Vert_{V} \eqsim \Vert \cdot \Vert_{V,{\rm eqv}}$ such
that the conditions of Theorem~\ref{our_theorem} are satisfied in these fitted norms when $c(\cdot,\cdot) \equiv 0$.

More specifically, for $c(\cdot,\cdot) \equiv 0$, we have
$\vert \cdot \vert_{Q}=\vert| \cdot \vert|_Q=\Vert q\Vert_{Q,{\rm eqv}}$
and $\bar Q=I$. 
If we define the fitted norm $\Vert \cdot \Vert_{V}$ by choosing
\begin{align}
|v|^2_V=a(v,v),
\end{align}
then \eqref{coerc_a} obviously holds. In addition, there exits a constant $\alpha_0$ such that
(see~\cite[Proposition~4.3.4]{Boffi2013mixed})
\begin{align}
\alpha_0 \Vert v\Vert^2_{V,{\rm eqv}}\le a(v,v)+\Vert Bv\Vert^2_{Q'}=\Vert v\Vert^2_{V}.
\end{align}
At the same time, under the conditions of Theorem~\ref{brezzi1}, the continuity of $a(\cdot, \cdot)$ and
$b(\cdot, \cdot)$ in the norms $\Vert \cdot\Vert_{V,{\rm eqv}}$ and $\Vert \cdot\Vert_{Q,{\rm eqv}}$, we have 
\begin{align}
\Vert v\Vert^2_{V}=a(v,v)+\Vert Bv\Vert^2_{Q'}\le C \Vert v\Vert^2_{V,{\rm eqv}}.
\end{align}
Thus, the fitted norm $\Vert\cdot\Vert_V$ is equivalent to the norm $\Vert\cdot\Vert_{V,{\rm eqv}}$,
and \eqref{inf_sup_b} is induced by~\eqref{small_inf_sup}.
\end{remark}
\begin{remark}
Note that under the conditions of the theorem
the coercivity of $a(\cdot,\cdot)$ on  ${\rm Ker}(B)$ in the (semi-) norms $\vert\cdot\vert_V$ and $\Vert\cdot\Vert_V$
are equivalent since $\vert v\vert_V=\Vert v\Vert_V$ for all $v\in {\rm Ker}(B)$. 
The inf-sup condition~\eqref{inf_sup_b}, however, uses
the seminorm $\vert \cdot \vert_Q$ instead of $\Vert \cdot \Vert_Q$ 
as in Brezzi's condition~\eqref{small_inf_sup}.
\end{remark}

\begin{proof}[Proof of Theorem~\ref{our_theorem}]
Demonstrating~\eqref{boundedness_A} is straightforward since 
\begin{subequations}
\allowdisplaybreaks
\begin{align*}
\mathcal{A}((\w;\r),(v;q)) =& 
a(\w,v)+b(v,\r)+b(\w,q)-c(\r,q) \nonumber \\
\le & \bar{C}_a \Vert \w\Vert_{V}\Vert v\Vert_{V} +\Vert v\Vert_V \Vert \r\Vert_Q  
+\Vert \w\Vert_V \Vert q\Vert_Q+ \Vert \r\Vert_Q \Vert q\Vert_Q \nonumber \\
\le & \bar{C}(\Vert \w\Vert_V+\Vert \r\Vert_Q)(\Vert v\Vert_V+\Vert q\Vert_Q)
\le 2 \bar{C} \Vert(\w;\r)\Vert_Y \Vert(v;q)\Vert_Y
\nonumber 
\end{align*}
\end{subequations}
with $\bar{C}:=\max\{\bar{C}_a,1\}$.

In order to prove~\eqref{inf_sup_A} for a positive constant $\delta$, which will be selected later, 
{and for a given arbitrary pair $(\w,\r)\in V\times Q$}, we choose  
\begin{equation}\label{v}
v :=\delta \w+\w_0
\end{equation}
where, {by Lemma \ref{equi:inf-sup}, $\w_0\in V$ can be chosen} such that 
\begin{subequations}\label{u0}
\begin{align}
b(\w_0,\r) =&\vert \r\vert_Q^2 \label{u0_1}, \\
\Vert \w_0\Vert_V \le& \underline{\beta}^{-1} \vert \r\vert_Q \label{u0_2},
\end{align}
\end{subequations}
and
\begin{equation}\label{q}
q  :=  -\delta \r  + \r_0 
\end{equation}
where 
\begin{equation}\label{p0}
\r_0:= \bar{Q}^{-1}B \w.
\end{equation}
Note that the existence of an element $\w_0$ satisfying~\eqref{u0} follows from~\eqref{inf_sup_b}.

Then we have
\begin{subequations}
\allowdisplaybreaks
\begin{align*}
\Vert v\Vert_V \le &
\Vert \delta \w\Vert_V +\Vert \w_0\Vert_V 
\le 
\delta \Vert \w\Vert_V +\underline{\beta}^{-1}\vert \r\vert_Q 
\le 
\delta \Vert \w\Vert_V +\underline{\beta}^{-1}\Vert \r\Vert_Q, \\
\Vert q\Vert_Q \le & \delta \Vert \r \Vert_Q + \Vert \r_0\Vert_Q = \delta \Vert \r\Vert_Q + (\bar{Q}^{-1}B\w,\bar{Q}^{-1}B\w)_Q^{1/2} = \delta \Vert \r\Vert_Q + \vert \w \vert_{b},
\nonumber 
\end{align*}
\end{subequations}
and, consequently, 
\begin{subequations}
\allowdisplaybreaks
\begin{align*}
\Vert (v;q)\Vert^2_Y = & \Vert v\Vert_V^2 +\Vert q\Vert_Q^2 
\le  2(\delta^2 +1)\Vert \w \Vert_V^2 +2(\underline{\beta}^{-2}+\delta^2)\Vert \r\Vert_Q^2.
\nonumber 
\end{align*}
\end{subequations}
Hence, it follows that
\begin{equation}\label{bound_est}
\Vert(v;q)\Vert_Y \le \left(2\max\{(\delta^2 +1),(\underline{\beta}^{-2}+\delta^2)\}\right)^{\frac{1}{2}} \Vert (\w;\r)\Vert_Y.
\end{equation}
 
Moreover, for the same choice of $v$ and $q$, we obtain
\begin{align*}
\allowdisplaybreaks
\mathcal{A}((\w;\r),(v;q)) & =  a(\w,\delta \w+\w_0)+b(\delta \w+\w_0,\r)-b(\w,\delta \r -\r_0)+c(\r,\delta \r -\r_0) 
 \\
& \ge  \delta a(\w,\w)+a(\w,\w_0)+\delta b(\w,\r)+b(\w_0,\r)-\delta b(\w,\r)+\delta c(\r,\r) 
\\
&+ \langle B\w, \bar{Q}^{-1}B\w \rangle_{Q'\times Q}-
c(\r, \bar{Q}^{-1}B\w)
\\
& \ge  \delta a(\w,\w) -\frac{1}{2} \epsilon^{-1}a(\w,\w)-\frac{1}{2}\epsilon a(\w_0,\w_0) +\vert \r\vert_Q^2
+\delta c(\r,\r)
\\
&+ \vert \w\vert^2_b-\frac{1}{2}c(\r, \r)
-\frac{1}{2}c( \bar{Q}^{-1}B\w,\bar{Q}^{-1}B\w)
\\
& \ge  \left(\delta-\frac{1}{2} \epsilon^{-1} \right) a(\w,\w) -\frac{1}{2}\epsilon \bar{C}_a \underline{\beta}^{-2} \vert \r\vert_Q^2 
+\vert \r\vert_Q^2 +\delta c(\r,\r)+{ \vert \w\vert^2_b -\frac{1}{2}c(\r,\r)-\frac{1}{2}\vert \w \vert^2_b}  
\\
& \ge  \left(\delta-\frac{1}{2} \epsilon^{-1} \right) \underline{C}_a \vert \w\vert_V^2 +\left(1-\frac{1}{2}\epsilon \bar{C}_a 
\underline{\beta}^{-2}\right)
\vert \r\vert_Q^2 +\delta c(\r,\r)  -\frac{1}{2}c(\r,\r) +{ \frac{1}{2}\vert \w\vert_b^2 }
\end{align*}
and, hence, for $\epsilon = \frac{1}{2}\bar{C}_a^{-1}\underline{\beta}^2$ and {$\delta=\max\{ \frac14\underline{C}_a^{-1}+\bar{C}_a\underline{\beta}^{-2}, \frac34\}$, we have }
\begin{align}
\allowdisplaybreaks
\mathcal{A}((\w;\r),(v;q)) & \ge 
\left(\delta - { \bar{C}_a} \underline{\beta}^{-2}\right) \underline{C}_a \vert \w\vert_V^2 
+{ \frac{1}{4}} \vert \r\vert_Q^2 +{ \left(\delta-\frac{1}{2}\right)} c(\r,\r) +{ \frac{1}{2}\vert \w\vert_{b}^2 }
\nonumber 
\\
& \ge 
{\frac{1}{4}}\left(
\Vert \w\Vert_V^2 +\Vert \r\Vert_Q^2 \right) = {\frac{1}{4}} \Vert(\w;\r)\Vert_Y^2
\label{coerc_est}.
\end{align}
Together, ~\eqref{bound_est} and~\eqref{coerc_est} imply the inf-sup 
condition~\eqref{inf_sup_A} which can equivalently be formulated as 
\begin{equation}\label{inf_sup_A_equiv}
\sup_{(v;q)\in Y} \frac{\mathcal{A}((\w;\r),(v;q))}{\Vert (v;q)\Vert_Y }\ge  \underline{\alpha}
\Vert (\w;\r)\Vert_Y \qquad \forall (\w;\r)\in Y
\end{equation}
because the supremum on the left-hand side of~\eqref{inf_sup_A_equiv} is bounded 
from below by 
$$
\frac{\mathcal{A}((\w;\r),(v;q))}{\Vert(v;q)\Vert_Y}
$$
if we insert any fixed $(v;q)$, in particular the choice we made and for which we proved
$$
\frac{\mathcal{A}((\w;\r),(v;q))}{\Vert(v;q)\Vert_Y} \ge \frac{\frac{1}{4} \Vert (\w;\r)\Vert_Y^2}
{(2\max \{(\delta^2+1),(\underline{\beta}^{-2}+\delta^2)\})^{1/2}\Vert (\w;\r)\Vert_Y}.
$$
\end{proof}

\begin{remark}
The statement of Theorem~\ref{our_theorem} remains valid if the norms $\Vert \cdot \Vert_Q$ and $\Vert \cdot \Vert_V$,
as defined in~\eqref{Q_norm} and~\eqref{V_norm}, are replaced with equivalent norms $\Vert \cdot\Vert_{Q,{\rm eqv}} \eqsim \Vert \cdot\Vert_Q$ and
$\Vert \cdot \Vert_{V,{\rm eqv}} \eqsim \Vert \cdot \Vert_V$, hence using $\Vert \cdot \Vert_{V,{\rm eqv}}$ in~\eqref{inf_sup_b} in this case.
The proof remains unchanged and the only difference in the final result is that the inf-sup constant $\underline{\alpha}$ in \eqref{inf_sup_A_equiv}
with respect to the (new) equivalent combined norm has to be scaled by the quotient of the constants in the norm equivalence
relation for the combined norms. For that reason, without loss of generality, we can use the fitted norms defined by~\eqref{Q_norm}
and~\eqref{V_norm} directly in the formulation of the theorem.
\end{remark}

\begin{remark}
{
An advantage of the new framework is that 
Theorem~\ref{our_theorem} provides sufficient conditions that are easy to verify in practice. These are 
given in form of an LBB-type condition on the basis of the proposed norm fitting technique, which allows, contrary to the technique 
in~\cite{Zulehner2011nonstandard}, to choose the norms subsequently, see also~Remark~\ref{remark_alternative}.
}
\end{remark}

\section{Applications of the framework}\label{sec:applications}

In this section four different classes of problems are analyzed by means of the proposed framework demonstrating its versatility 
and ease of application. 
%
Here, we use bold letters to denote  
vector-valued functions and the spaces to which they belong which means that we identify certain non-bold symbols from the abstract 
framework in the previous section with bold symbols, e.g., $v = \bv$.
To prove stability of 
the exemplified mixed variational 
formulations, we assume that 
proper boundary conditions are imposed.
In certain cases, we will also make use of the following classical inf-sup conditions, see~\cite{Brezzi1974},  
also \cite{Brezzi1991mixed,Boffi2013mixed}, for the pairs of spaces 
$(\bV,Q)$:  
there exist 
constants $\beta_d$ and $\beta_s$ such that
%
\begin{align}\label{beta_d}
	\inf_{q\in Q}
\sup_{\bv \in \bV}
\frac{({\rm div} \bv,q)}{\Vert \bv\Vert_{\rm div} \Vert q\Vert}
\geq \beta_d  > 0 ,
	\end{align}
%
\begin{align}\label{beta_s}
	\inf_{q\in Q}
\sup_{\bv \in \bV}
\frac{({\rm div} \bv,q)}{\Vert \bv\Vert_1 \Vert q\Vert}
\geq \beta_s  > 0,
	\end{align}
where the norms $\Vert \cdot \Vert_{\rm div}$, $\Vert \cdot\Vert_1$ and $\Vert \cdot \Vert$ denote the 
standard $\boldsymbol H({\rm div})$, $\boldsymbol H^1$ and $L^2$ norms and $(\cdot,\cdot)$ is the $L^2$-inner product. 


\subsection{Generalized Poisson and generalized Stokes equations}
\begin{example}\label{example1}
The first example, see~\cite{boon2020robust}, is the following mixed variational problem resulting from a weak formulation of a 
generalized Poisson equation: find $(\bu,p)\in \bH({\rm div},\Omega)\times L^2(\Omega)$ such that 
\begin{subequations}
\begin{align*}
(\bu,\bv)+(p,{\rm div}\bv)& = 0,  \qquad \forall \bv\in \bH({\rm div},\Omega), \\
({\rm div} \bu,q)-t (p,q) & = -(f,q),\qquad \forall q\in L^2(\Omega),
\end{align*}
\end{subequations}
where $t\ge 0$ is a parameter. 

The bilinear forms {generating} $\mathcal{A}((\cdot ;\cdot),(\cdot ;\cdot))$ 
are given by 
$$
a(\bu,\bv):= (\bu,\bv),\quad b(\bv,q):= ({\rm div}\bv,q), \quad c(p,q)=t(p,q), 
\quad \forall \bu,\bv \in \bV, \forall p,q\in Q,
$$
where
$Q:=L^2(\Omega)$, $\bV:=\bH({\rm div},\Omega)$. Using the norm fitting technique, we define $\vert \cdot \vert_Q$, 
$\vert \cdot \vert_V$ by
\begin{equation*}
\vert q\vert_Q^2 := (q,q) \quad \forall q\in Q \quad \text{and} \quad 
\vert \bv \vert_V^2 :=(\bv,\bv) \qquad \forall \bv \in \bV.
\end{equation*}
Obviously, $B:={\rm div}: \bV\rightarrow Q'$, and~\eqref{Q_norm} and 
\eqref{V_norm} take the form
\begin{equation*}
\Vert q\Vert_Q^2 = \vert q\vert_Q^2 + c(q,q)=(q,q)+t(q,q)=((1+t)q,q)=\langle (1+t)I q,q\rangle_{Q'\times Q},
\end{equation*}
\begin{equation*}
\Vert \bv\Vert_V^2 = \vert \bv\vert_V^2+\vert \bv\vert_b^2 = (\bv,\bv)
+ \langle {\rm div} \bv, \frac{1}{(1+t)}I {\rm div} \bv\rangle_{Q' \times Q} 
= (\bv,\bv) + \frac{1}{(1+t)}({\rm div} \bv, {\rm div}\bv),
\end{equation*}
respectively. 
Since $a(\bv,\bv)=(\bv,\bv)=\vert \bv\vert_V^2$ for all $\bv\in \bV$, condition 
\eqref{coerc_a} in Theorem~\ref{our_theorem} is satisfied with $\underline{C}_a=1$. 

Finally, we have to verify condition~\eqref{inf_sup_b} in Theorem~\ref{our_theorem}, namely
\begin{equation*}
\sup_{\bv \in V} \frac{({\rm div}\bv,q)}{(\Vert \bv\Vert^2+\frac{1}{(1+t)} \Vert {\rm div} \bv\Vert^2)^{1/2}}
\ge \underline{\beta} \vert q\vert_Q = : \underline{\beta}\Vert q\Vert, \qquad \forall q\in Q
\end{equation*}
which follows directly from the classical inf-sup condition~\eqref{beta_d} on the spaces $(\bV,Q)$ { since $t \ge 0$}.

As a result, we obtain that the preconditioner
\begin{align*}
\mathcal{B} := \begin{bmatrix}
(I-(1+t)^{-1} \nabla \rm div )^{-1} & \\
& ((1+t)I)^{-1}
\end{bmatrix}
\end{align*}
is norm-equivalent for the combined norm, cf.~\cite{mardal2011preconditioning}. To solve the $\bH(\rm div)$ subproblem, one can use various 
preconditioners, see, e.g.,~\cite{Hiptmair.R;Xu.J2007a,kraus2016preconditioning}, multigrid, see, e.g.,~\cite{vassilevski1996preconditioning,arnold1997preconditioning}, 
and domain decomposition methods~\cite{Toselli2005domain}.

\begin{remark}
Note that in this example as well as the ones which follow, the perturbation term $c(\cdot,\cdot)$, due to the presence of various parameters, 
can dominate the problem. In this situation stability often cannot be proven using Theorem~\ref{brezzi2} because  
$\Vert q \Vert_Q$ has to bound $c(q,q)^{\frac{1}{2}}$ which for dominating perturbation conflicts satisfying the classical LBB condition~\eqref{small_inf_sup} uniformly. 
A way to overcome this problem is to work either with the Braess inf-sup condition, see Theorem~\ref{braess_theorem}, 
or with the Babu\v{s}ka inf-sup condition, see Theorem~\ref{thm:1}, 
or, alternatively use the simpler to verify inf-sup condition provided in Theorem~\ref{our_theorem}.
\end{remark}

\end{example} 

\begin{example}\label{example2}
The second example we consider is taken from~\cite{lee2017parameter}. 
Its mixed variational formulation reads: find $(\bu,p)\in \bH_0^1(\Omega)\times (H^1(\Omega) \cap L^2_0(\Omega))$ such that 
\begin{subequations}
\begin{align*}
(\nabla \bu, \nabla \bv)-(p,{\rm div}\bv)& = (\bbf,\bv),  \qquad \forall \bv\in \bH_0^1(\Omega), \\
-({\rm div} \bu,q)- (\kappa \nabla p,\nabla q) & = (g,q),\qquad \forall q\in H^1(\Omega) \cap L^2_0(\Omega), 
\end{align*}
\end{subequations}
where $\kappa\ge 0$ is a {parameter}. 

The bilinear forms defining $\mathcal{A}((\cdot ;\cdot),(\cdot ;\cdot))$ 
here are given by 
$$
a(\bu,\bv):= (\nabla\bu,\nabla\bv),\quad b(\bv,q):= -({\rm div}\bv,q), \quad c(p,q)=(\kappa\nabla p,\nabla q), \quad \forall \bu,\bv\in \bV, \forall p,q\in Q.
$$
In this example, we set  
$Q:=H^1(\Omega) \cap L^2_0(\Omega)$, $\bV:=\bH_0^1(\Omega)$, and $\vert \cdot \vert_Q$, 
$\vert \cdot \vert_V$ to be 
\begin{equation*}
\vert q\vert_Q^2 := (q,q) \quad \forall q\in Q \quad \text{and} \quad 
\vert \bv \vert_V^2 :=(\nabla \bv,\nabla\bv) \qquad \forall \bv \in \bV.
\end{equation*}
Then the operator $B$ is defined by $B: \bV\rightarrow Q'$, $B:=-{\rm div}$ and the norm splittings~\eqref{Q_norm} and~\eqref{V_norm} are given by 
\begin{equation}\label{ex2}
\Vert q\Vert_Q^2 = \vert q\vert_Q^2 + c(q,q)=(q,q)+(\kappa \nabla q,\nabla q)
= \langle \bar{Q}q,q\rangle_{Q'\times Q}
\end{equation}
and 
\begin{equation*}
\Vert \bv\Vert_V^2 = \vert \bv\vert_V^2+\vert \bv\vert_b^2 = (\nabla \bv,\nabla\bv)
+ \langle B \bv, \bar{Q}^{-1}B\bv \rangle_{Q' \times Q}. 
\end{equation*}
Condition~\eqref{coerc_a} is automatically satisfied with $\underline{C}_a=1$. To show 
\eqref{inf_sup_b} we first note that using~\eqref{ex2} we obtain 
$$
\begin{array}{rl}
\langle B\bv,\bar{Q}^{-1}B\bv\rangle_{Q'\times Q} & = \Vert B \bv\Vert_{Q'}^2 =\left( \sup_{q\neq 0} \frac{b(\bv,q)}{\Vert q\Vert_Q}\right)^2 = 
\left( \sup_{q\neq 0} \frac{({\rm div}\bv,q)}{\Vert q\Vert_Q}\right)^2\le \left( \sup_{q\neq 0} \frac{\Vert{\rm div}\bv\Vert \Vert q\Vert}{\Vert q\Vert_Q}\right)^2 
\\
& \le 
({\rm div} \bv, {\rm div} \bv).
\end{array}
$$
Thus, 
\begin{equation}\label{v_norm_est}
\Vert \bv\Vert_V^2 = (\nabla \bv,\nabla \bv) + \langle B\bv,\bar{Q}^{-1}B\bv\rangle_{Q'\times Q} 
\le (\nabla \bv,\nabla \bv) + ({\rm div}\bv, {\rm div}\bv)\le 2(\nabla \bv,\nabla \bv) 
\le 2 \Vert \bv\Vert_1^2.
\end{equation}
Now, we choose $\bv_0$ such that $-{\rm div} \bv_0 = q$ and hereby obtain from the Stokes 
inf-sup condition~\eqref{beta_s} the estimate $\Vert \bv_0\Vert_1 \le \frac{1}{\beta_s} \Vert q\Vert$, and, finally,  
\begin{align*}
\sup_{\bv \in V} \frac{b(\bv,q)}{\Vert \bv\Vert_V} \ge 
\frac{b(\bv_0,q)}{\Vert \bv_0\Vert_V} = \frac{\Vert q\Vert^2}{\Vert \bv_0\Vert_V} 
\ge \frac{1}{\sqrt{2}}\frac{\Vert q\Vert^2}{\Vert \bv_0\Vert_1} 
\ge \frac{\beta_s}{\sqrt{2}} \frac{\Vert q\Vert^2}{ \Vert q\Vert}=: \underline{\beta}\Vert q\Vert 
=\underline{\beta} \vert q\vert_Q, \qquad \forall q\in Q.
\end{align*}

The induced norm-equivalent preconditioner in this example reads as
\begin{align*}
\mathcal{B} := \begin{bmatrix}
(-\Delta    - \nabla (I -  {\rm div}\kappa \nabla)^{-1} {\rm div} )^{-1} & \\
& (I -  {\rm div}\kappa \nabla)^{-1}
\end{bmatrix}
\eqsim 
 \begin{bmatrix}
-\Delta^{-1}     & \\
& (I -  {\rm div}\kappa \nabla)^{-1}
\end{bmatrix},
\end{align*}
where the equivalence is due to~\eqref{v_norm_est}. 
\end{example}

{
\subsection{Stokes Darcy problem}

\begin{example}\label{example_stokes_darcy}
Let  $\Omega=\Omega_S\cup  \Omega_D$ and $\Gamma=\partial \Omega_S\cap \partial\Omega_D$.  We assume that $\Gamma_{S}^{D} \cup \Gamma_{S}^{N} \cup \Gamma$ forms a disjoint decomposition of $\partial \Omega_{S}$ and, similarly, $\Gamma_{D}^{D} \cup \Gamma_{D}^{N} \cup \Gamma$ is a partition of $\partial \Omega_{D}$. Denote 
$$
H^1_{\Gamma_{i}^{D} }(\Omega_i)=\left\{w\in H^1(\Omega_i):  w|_{\Gamma_{i}^{D}}=0\right\},~~i=S~\hbox{or}~D. 
$$
The classical formulation of Stokes Darcy problem follows \cite{discacciati2002mathematical}: Find $\left(\boldsymbol{u}, p_{S}, p_{D}\right) \in \boldsymbol H^1_{\Gamma_{S}^{D} }(\Omega_S) \times L^2(\Omega_S) \times H^1_{\Gamma_{D}^{D} }(\Omega_D)$ such that
$$
\begin{aligned}
\left(2 \mu \boldsymbol{\epsilon}\left(\boldsymbol{u}\right), \boldsymbol{\epsilon}\left(\boldsymbol{v}\right)\right)_{\Omega_{S}}+\beta_{\tau}\left(\boldsymbol{\tau} \cdot \boldsymbol{u},\boldsymbol{\tau} \cdot \boldsymbol{v}\right)_{\Gamma}&&\\
-\left(p_{S}, \nabla \cdot \boldsymbol{v}\right)_{\Omega_{S}}+\left(p_{D}, \boldsymbol{n} \cdot \boldsymbol{v}\right)_{\Gamma} &=\left(\boldsymbol{f}_{S}, \boldsymbol{v}\right)_{\Omega_{S}}, & & \forall \boldsymbol{v} \in \boldsymbol H^1_{\Gamma_{S}^{D} }(\Omega_S),\\
-\left(\nabla \cdot \boldsymbol{u}, q_{S}\right)_{\Omega_{S}} &=0,& & \forall q_{S} \in L^2(\Omega_S),  & & \\
\left(\boldsymbol{n} \cdot \boldsymbol{u}, q_{D}\right)_{\Gamma}-\left(\kappa \nabla p_{D}, \nabla q_{D}\right)_{\Omega_{D}} &=\left(f_{D}, q_{D}\right)_{\Omega_{D}},& & \forall q_{D} \in H^1_{\Gamma_{D}^{D} }(\Omega_D),
\end{aligned}
$$
where $\boldsymbol{n}:=\boldsymbol{n}_{S}$ is the outer normal of the Stokes domain and $\boldsymbol{\tau}:=\boldsymbol{I}-(\boldsymbol{n} \otimes \boldsymbol{n})$ is the projection onto the tangent bundle of the interface $\Gamma$.

The bilinear forms defining $\mathcal{A}((\cdot ;\cdot),(\cdot ;\cdot))$ 
here are given by 
\begin{align*}
a(\bu,\bv)&:= \left(2 \mu \boldsymbol{\epsilon}\left(\boldsymbol{u}\right), \boldsymbol{\epsilon}\left(\boldsymbol{v}\right)\right)_{\Omega_{S}}+\beta_{\tau}\left(\boldsymbol{\tau} \cdot \boldsymbol{u},\boldsymbol{\tau} \cdot \boldsymbol{v}\right)_{\Gamma},  \quad \forall \bu,\bv\in \bV,\\
\quad b(\bv, \boldsymbol q)&:= -\left(q_{S}, \nabla \cdot \boldsymbol{v}\right)_{\Omega_{S}}+\left(q_{D}, \boldsymbol{n} \cdot \boldsymbol{v}\right)_{\Gamma}, \quad \forall \bv\in \bV, \forall \boldsymbol q\in \boldsymbol Q, \\
\quad c(\boldsymbol p,\boldsymbol q)&:=\left(\kappa \nabla p_{D}, \nabla q_{D}\right)_{\Omega_{D}} ,\quad \forall \boldsymbol p,\boldsymbol q\in \boldsymbol Q,
\end{align*}
where $\bV=\boldsymbol H^1_{\Gamma_{S}^{D} }(\Omega_S), \boldsymbol Q=L^2(\Omega_S) \times H^1_{\Gamma_{D}^{D} }(\Omega_D)$ and $\boldsymbol p=(p_S, p_D), \boldsymbol q=(q_S, q_D) $. 
We fix $\vert \cdot \vert_Q$, $\vert \cdot \vert_V$ to be 
$$
|\boldsymbol q|_Q^2:=(2\mu)^{-1} (q_S, q_S) _{\Omega_S}+(2\mu)^{-1}\Vert  q_D\Vert _{-\frac{1}{2},\Gamma}^2, \quad\forall \boldsymbol q\in \boldsymbol Q,
$$
$$
\vert  \bv\vert _V^2:=\left(2 \mu \boldsymbol{\epsilon}\left(\boldsymbol{v}\right), \boldsymbol{\epsilon}\left(\boldsymbol{v}\right)\right)_{\Omega_{S}}+\beta_{\tau}\left(\boldsymbol{\tau} \cdot \boldsymbol{v},\boldsymbol{\tau} \cdot \boldsymbol{v}\right)_{\Gamma}, \quad \forall\bu\in \bV.
$$
Noting that $a(\bv,\bv)= \vert \bv\vert^2_{\bV}$ for all $\bv\in \bV$, that is,  
~\eqref{coerc_a} is satisfied with $\underline{C}_a=1$.  In addition, we have 
\begin{align*}
\Vert \bm q\Vert_Q^2 &:=\vert \bm q\vert_Q^2+c(\bm q, \bm q)=
\langle\bar Q  \bm q, \bm q\rangle_{Q'\times Q},
~ \forall \bm q\in \bQ, \\ 
\Vert \bv \Vert_{\bV}^2 &:= \vert \bv \vert_{\bV}^2+\langle 
B \bv, \bar Q ^{-1} B\bv\rangle_{Q'\times Q}= \vert \bv \vert_{\bV}^2+\vert \bv\vert_b^2, ~ \forall \bv \in \bV.
\end{align*}
The continuity of $B$ is shown by the following calculation, utilizing the Cauchy-Schwarz inequality and a trace inequality:
\begin{equation}
\begin{aligned}
\left\langle B \boldsymbol{v},\boldsymbol q \right\rangle_{\boldsymbol Q'\times \boldsymbol Q} &=-\left(\nabla \cdot \boldsymbol{v}, p_{S}\right)_{\Omega_{S}}+\left(\boldsymbol{n} \cdot \boldsymbol{v}, q_{D}\right)_{\Gamma} \\
& \leq\Vert\nabla \cdot \boldsymbol{v}\Vert_{\Omega_{S}}\Vert q_{S}\Vert_{\Omega_{S}}+\Vert\boldsymbol{n} \cdot \boldsymbol{v}\Vert_{\frac{1}{2}, \Gamma}\Vert q_{D}\Vert_{-\frac{1}{2}, \Gamma} \\
& \lesssim\Vert\boldsymbol{\epsilon}\left(\boldsymbol{v}\right)\Vert_{\Omega_{S}}\left(\Vert q_{S}\Vert_{\Omega_{S}}+\Vert q_{D}\Vert_{-\frac{1}{2}, \Gamma}\right) \\
& \lesssim(2 \mu)^{\frac{1}{2}}\Vert\boldsymbol{\epsilon}\left(\boldsymbol{v}\right)\Vert_{\Omega_{S}}(2 \mu)^{-\frac{1}{2}}\left(\Vert q_{S}\Vert_{\Omega_{S}}^{2}+\Vert q_{D}\Vert_{-\frac{1}{2}, \Gamma}^{2}\right)^{\frac{1}{2}} \\
&\le \vert\boldsymbol{v}\vert_{\bV}\left|q\right|_{\boldsymbol Q}\le \vert\boldsymbol{v}\vert_{V}\Vert q\Vert_{\boldsymbol Q}
\end{aligned}
\end{equation}
Furthermore, by \eqref{seminorm_b}, we have 
\begin{equation}
\begin{aligned}
\vert \bv\vert_b=\Vert B\bv\Vert_{\boldsymbol Q'}=\sup_{\boldsymbol q\in \boldsymbol Q}\frac{\left\langle B \boldsymbol{v},\boldsymbol q \right\rangle_{\boldsymbol Q'\times \boldsymbol Q} }{\Vert q\Vert_{\boldsymbol Q}}\le \vert\boldsymbol{v}\vert_{\bV}, \quad \forall \bv\in \bV,
\end{aligned}
\end{equation}
and $\Vert \cdot\Vert_{\bV}$ is equivalent to $\vert \cdot\vert_{\bV}$, namely 
\begin{equation}\label{vvequi}
\Vert \bv\Vert_{\bV}\eqsim  \vert\boldsymbol{v}\vert_{\bV},\quad \forall \bv\in \bV.
\end{equation}
Now we show that~\eqref{inf_sup_b} is satisfied. By \eqref{vvequi}, it suffices to show the following inf-sup condition of~$B$:
there exists a constant $\underline{\beta}>0$ such that 
\begin{equation}\label{inf_sup_bv}
\sup_{\substack{\bv\in \bV\\ \bv\neq 0}} 
\frac{b(\bv,\boldsymbol q)}{\vert \bv\vert_{\bV}}\ge \underline{\beta}  \vert \boldsymbol q\vert_{\boldsymbol Q}, \quad \forall \boldsymbol q\in \boldsymbol Q.
\end{equation} 
For any given $\boldsymbol q=\left(q_{S}, q_{D}\right)$, let $\boldsymbol{v}^{S} \in$ $\boldsymbol{H}^{1}\left(\Omega_{S}\right)$ be constructed, using the Stokes inf-sup condition ~\eqref{beta_s}, such that
\begin{equation}
\boldsymbol{v}^{S}|_{\Gamma}=0, \quad \nabla \cdot \boldsymbol{v}^{S}=-q_{S}, \quad \Vert\boldsymbol{\epsilon}\left(\boldsymbol{v}^S\right)\Vert_{\Omega_{S}} \le \beta_s\Vert q_{S}\Vert_{\Omega_{S}}.
\end{equation}
On the other hand, let $\phi \in H^{\frac{1}{2}}(\Gamma)$ be the Riesz representative of $\left.q_{D}\right|_{\Gamma} \in H^{-\frac{1}{2}}(\Gamma)$. We then define $\boldsymbol{v}^{D} \in \boldsymbol{H}^{1}\left(\Omega_{S}\right)$ as the bounded extension that satisfies
\begin{equation}
\boldsymbol{v}^{D}|_{\Gamma}=\phi \boldsymbol n, \quad \nabla \cdot \boldsymbol{v}^{D}=0, \quad \Vert\boldsymbol{\epsilon}\left(\boldsymbol{v}^D\right)\Vert_{\Omega_{S}} \le \beta_0 \Vert \phi\Vert_{\frac12,\Gamma}= \beta_0 \Vert q_{D}\Vert_{-\frac12,\Gamma}.
\end{equation}
We now set the test function $\boldsymbol{v}_0:=(2 \mu)^{-1}\left(\boldsymbol{v}^{S}+\boldsymbol{v}^{D}\right)$. Noting that $\boldsymbol{\tau} \cdot \boldsymbol{v}_0=0$ on $\Gamma$, this function satisfies
$$
\begin{aligned}
b(\bv_0, \boldsymbol q)&=-(2 \mu)^{-1}\left(\nabla \cdot \boldsymbol{v}^{S}, q_{S}\right)_{\Omega_{D}}+(2 \mu)^{-1}\left(\boldsymbol{n} \cdot \boldsymbol{v}^{D}, q_{D}\right)_{\Gamma} \\
&=(2 \mu)^{-1}\Vert q_{S}\Vert_{\Omega_{S}}^{2}+(2 \mu)^{-1}\Vert q_{D}\Vert_{-\frac{1}{2}, \Gamma}^{2}=\vert \boldsymbol q\vert_{\boldsymbol Q}^2, \\
\vert \boldsymbol{v}_0\vert_{\bV} &=(2 \mu)^{\frac{1}{2}}\Vert \boldsymbol{\epsilon}\left((2 \mu)^{-1}\left(\boldsymbol{v}^{S}+\boldsymbol{v}^{D}\right)\right)\Vert_{\Omega_{S}} \\
& \leq(2 \mu)^{-\frac{1}{2}}\left(\Vert \boldsymbol{\epsilon}\left(\boldsymbol{v}^{S}\right)\Vert_{\Omega_{S}}+\Vert \boldsymbol{\epsilon}\left(\boldsymbol{v}^{D}\right)\Vert_{\Omega_{S}}\right) \\
& \lesssim(2 \mu)^{-\frac{1}{2}}\left(\Vert q_{S}\Vert_{\Omega_{S}}+\Vert q_{D}\Vert_{-\frac{1}{2}, \Gamma}\right)=\vert \boldsymbol q\vert_{\boldsymbol Q}.
\end{aligned}
$$
Hence, condition \eqref{inf_sup_bv} is fulfilled.
\end{example}

}
\subsection{Vector Laplace equation}
\begin{example}\label{example3}
We consider the following mixed variational formulation of the vector Laplace equation \cite{arnold2006finite,hong2021extended}: 
find $\bm p \in \bH_0({\rm curl},\Omega), \bu\in \bH_0({\rm div},\Omega)$, such that 
\begin{subequations}
\begin{align*}
(\alpha \bm p, \bm q)-(\bu, {\rm curl} \bm q)& = 0,  \qquad \forall \bm q\in \bH_0({\rm curl},\Omega), \\
-({\rm curl}\bm p, \bv)-({\rm div}\bu, {\rm div}\bv) & = (f,\bm v),\qquad \forall \bv\in \bH_0({\rm div},\Omega),
\end{align*}
\end{subequations}
where $\alpha$ is a positive scalar. 
Here, $\bH_0({\rm curl},\Omega)=\{\bm q\in \bm L^2(\Omega): {\rm curl} {\bm q} \in \bm L^2(\Omega),\, \bm q \times \bm n =0 \text{ on } \partial \Omega\}$ 
and $\bH_0({\rm div},\Omega)=\{\bm v\in \bm L^2(\Omega) : {\rm div} {\bm v} \in L^2(\Omega),\,\bm v \cdot \bm n =0 \text{ on } \partial \Omega\}$.
We rewrite the above equations as  
\begin{subequations}
\begin{align*}
({\rm div}\bu, {\rm div}\bv )+({\rm curl}\bm p, \bv)& = -(f,\bv),\qquad \forall \bv\in \bH_0({\rm div},\Omega),\\
(\bu, {\rm curl} \bm q)-(\alpha \bm p, \bm q)& = 0,  \qquad \forall \bm q\in \bH_0({\rm curl},\Omega).
\end{align*}
\end{subequations}
The bilinear forms that define $\mathcal{A}((\cdot ;\cdot),(\cdot ;\cdot))$ are 
\begin{align*}
a(\bu,\bv)&:= ({\rm div}\bu, {\rm div}\bv ),\qquad \forall \bu,\bv\in \bV,\\
b(\bv, \bm p)&:= ({\rm curl}\bm p, \bv),\qquad \forall\bv\in \bV,\forall \bm p\in \bQ, \\
c(\bm p, \bm q)&:= (\alpha \bm p, \bm q),\qquad \forall \bm p,\bm q\in \bQ,
\end{align*}
where $\bV=\bH_0({\rm div},\Omega), \bQ=\bH_0({\rm curl},\Omega)$. We fix 
$\vert \cdot \vert_Q$ and 
$\vert \cdot \vert_V$ to be 
\begin{align*}
\vert \bm q\vert_Q^2 &:=((\alpha+1) {\rm curl}  \bm q,  {\rm curl}\bm q),
\qquad \forall \bm q\in \bQ, \\ 
\vert \bv \vert_V^2 &:= ({\rm div}\bv, {\rm div}\bv ), \qquad \forall \bv \in \bV.
\end{align*}
As in the previous examples, $a(\bv,\bv)= \vert \bv\vert_V^2$ for all $\bv\in \bV$, that is,  
~\eqref{coerc_a} is satisfied with $\underline{C}_a=1$.  In addition, noting that $B:={\rm curl^*}: \bV\rightarrow \bQ'$, we have 
\begin{align*}
\Vert \bm q\Vert_Q^2 &:=\vert \bm q\vert_Q^2+c(\bm q, \bm q)=((\alpha+1){\rm curl}  \bm q,  {\rm curl}\bm q)+(\alpha \bm q, \bm q)=
\langle\bar Q  \bm q, \bm q\rangle_{Q'\times Q},
~ \forall \bm q\in \bQ, \\ 
\Vert \bv \Vert_V^2 &:= \vert \bv \vert_V^2+\langle 
B \bv, \bar Q ^{-1} B\bv\rangle_{Q'\times Q}
=({\rm div}\bv, {\rm div}\bv )+ \left((\alpha I+{\rm curl^*}(\alpha+1) {\rm curl})^{-1} {\rm curl^*} \bv, {\rm curl^*} \bv\right), ~ \forall \bv \in \bV.
\end{align*}
Now we show that~\eqref{inf_sup_b} is satisfied.  For any $\bm q\in \bQ$, choose $\bv_0={\rm curl} \bm q\in \bV$
to obtain
\begin{align*}
\Vert \bv_0 \Vert_V^2&=({\rm div} \, {\rm curl} \bm q, {\rm div}\, {\rm curl} \bm q )+\left((\alpha I+ {\rm curl^*}(\alpha+1){\rm curl})^{-1} {\rm curl^*}  {\rm curl} \bm q, {\rm curl^*}  {\rm curl} \bm q\right)\\
&=\left((\alpha I+ {\rm curl^*}(\alpha+1)  {\rm curl})^{-1} {\rm curl^*}   {\rm curl} \bm q, {\rm curl^*}   {\rm curl} \bm q\right)\le  (\bm q,(\alpha+1)^{-1} {\rm curl^*}   {\rm curl} \bm q)\\
&= ( (\alpha+1)^{-1} {\rm curl} \bm q,  {\rm curl} \bm q)
\end{align*} 
and 
\begin{equation}\label{ex3}
\begin{array}{rl}
\sup_{\bv \in \bV} \frac{b(\bv,\bm q)}{\Vert \bv\Vert_V} &\ge 
\frac{b(\bv_0,\bm q)}{\Vert \bv_0\Vert_V} = \frac{(  {\rm curl} \bm q,  {\rm curl} \bm q)}{\Vert \bv_0\Vert_V} 
\ge  \frac{(  {\rm curl} \bm q,  {\rm curl} \bm q)}{((\alpha+1)^{-1} {\rm curl} \bm q,  {\rm curl} \bm q)^{\frac{1}{2}}}= \vert \bm q\vert_Q.
\end{array}
\end{equation}
\begin{remark} 
Consider $\bH({\rm div}_0, \Omega)=\{\bv\in \bH({\rm div}, \Omega): {\rm div} \bv=0 \}$. 
Then for any $\alpha>0$ and $\bv\in \bH_0({\rm curl}, \Omega)\cap  \bH({\rm div}_0, \Omega)$, we have 
$$
((\alpha I+(\alpha+1){\rm curl^*curl}) \bv, \bv)\le (c_P\alpha+(\alpha+1)) ({\rm curl^*curl} \bv, \bv) \le c_1((\alpha+1)  {\rm curl^*curl} \bv, \bv), 
$$
where 
$c_P$ denotes the Poincar\'e constant for the $\rm curl$ operator and $c_1=c_P+1$. Hence, 
$$
((\alpha I+(\alpha+1){\rm curl^*curl})^{-1} \bm f, \bm f)\ge c_1^{-1}((\alpha+1)^{-1} ({\rm curl^*curl})^{-1} \bm f, \bm f), 
$$
for any $\bm f\in (\alpha I+(\alpha+1){\rm curl^*curl}) \left(\bH_0({\rm curl}, \Omega)\cap  \bH({\rm div}_0, \Omega)\right)$. 
Now, by using the Helmholtz decomposition $\bv={\rm curl} \bm w +\nabla z$ of $\bv$ and choosing $\bm f= {\rm curl^*} \bv$ we obtain 
\begin{equation}\label{Poincarecurl}
c_1^{-1}((\alpha+1)^{-1}{\rm curl} \bm w, {\rm curl} \bm w)\le  \left((\alpha I+(\alpha+1){\rm curl^*curl})^{-1} {\rm curl^*} \bv, {\rm curl^*} \bv\right).
\end{equation}
On the other hand, the Poincare's inequality for vector Laplacian (see Theorem 2.2 in \cite{arnold2006finite}) implies 
\begin{align}\label{Poincarevec1}
(\nabla z, \nabla z) \le c_{P_v}\left( ({\rm div}(\nabla z), {\rm div}(\nabla z))+({\rm curl}(\nabla z), {\rm curl}(\nabla z))\right)=c_{P_v} ({\rm div}\bv, {\rm div}\bv),
\end{align}
where $c_{P_v}$ denotes the Poincar\'e constant for vector Laplacian.
By multiplying \eqref{Poincarevec1} with $c_1^{-1}(\alpha+1)^{-1}$ we obtain
\begin{align}\label{Poincarevec}
c_1^{-1}(\alpha+1)^{-1}(\nabla z, \nabla z) \le c_1^{-1}(\alpha+1)^{-1}c_{P_v} ({\rm div}\bv, {\rm div}\bv) \le c_1^{-1}c_{P_v} ({\rm div}\bv, {\rm div}\bv).
\end{align}
Combining \eqref{Poincarecurl} and \eqref{Poincarevec} and noting that $(\bv, \bv)=({\rm curl} \bm w, {\rm curl} \bm w)+(\nabla z, \nabla z)$, we have
\begin{align*}
c_1^{-1}(\alpha+1)^{-1}\Vert \bv \Vert^2 \leq c_1^{-1}c_{P_v} \Vert{\rm div}\bv\Vert^2 +\left((\alpha I+(\alpha+1){\rm curl^*curl})^{-1} {\rm curl^*} \bv, {\rm curl^*} \bv\right).
\end{align*}
Next we add $c_1^{-1}\Vert{\rm div}\bv\Vert^2 $ to both sides and get
\begin{align*}
c_1^{-1}\big((\alpha+1)^{-1}\Vert \bv \Vert^2 + \Vert{\rm div}\bv\Vert^2\big) \leq c_1^{-1}(c_{P_v}+1) \Vert{\rm div}\bv\Vert^2  
+\left((\alpha I+(\alpha+1){\rm curl^*curl})^{-1} {\rm curl^*} \bv, {\rm curl^*} \bv\right).
\end{align*}
By multiplying with $c_1$ and setting $c_2: =\max\left\lbrace  c_1,(c_{P_v}+1)\right\rbrace$ it follows that
\begin{align*}
(\alpha+1)^{-1}\Vert \bv \Vert^2 + ({\rm div}\bv, {\rm div}\bv) \leq c_2\Vert \bv \Vert_V^2.
\end{align*}
Moreover, we have
\begin{align*}
\Vert \bv \Vert_V^2 = \vert \bv \vert_V^2+\langle 
B \bv, \bar Q ^{-1} B\bv\rangle_{Q'\times Q}
&=({\rm div}\bv, {\rm div}\bv )+ \left((\alpha I+{\rm curl^*}(\alpha+1) {\rm curl})^{-1} {\rm curl^*} \bv, {\rm curl^*} \bv\right)
\\
&\le ({\rm div}\bv, {\rm div}\bv )+ \left((\alpha+1)^{-1} \bv,\bv\right).
\end{align*}
Therefore we conclude that $\Vert \bv \Vert_V^2$ and $({\rm div}\bv, {\rm div}\bv)+((\alpha+1)^{-1}\bv, \bv)$ are equivalent.

\end{remark}

The corresponding norm-equivalent preconditioner  is then given by
\begin{align*}
\mathcal{B} :=
\begin{bmatrix}
((\alpha+1)^{-1}I - \nabla{ \rm div})^{-1} & \\
& (\alpha I + (\alpha+1){\rm curl}^* {\rm curl})^{-1}
\end{bmatrix}.
\end{align*}
To solve the $\bH(\rm curl)$ subproblem, 
one can use the preconditioner proposed in~\cite{Hiptmair.R;Xu.J2007a}, 
the multigrid method in~\cite{arnold1997preconditioning}, or domain decomposition methods~\cite{Toselli2005domain}. 

\end{example}

\subsection{Poromechanics}

\begin{example}\label{example4}
The two-field formulation of the quasi-static Biot's consolidation model  
after semidiscretization in time by the implicit Euler method as 
studied in~\cite{lee2017parameter,Adler2018robust}, 
reads: find $(\bu,p_F)\in \bH_0^1(\Omega)\times H_0^1(\Omega)$ such that
\begin{subequations}
\begin{align*}
(\varepsilon( \bu), \varepsilon(\bv))+\lambda({\rm div} \bu, {\rm div}\bv)-\alpha
(p_F,{\rm div}\bv)& = (\bbf,\bv),  \qquad \forall \bv\in \bH_0^1(\Omega), \\
-\alpha({\rm div} \bu,q_F)-{c_0}(p_F,q_F)
-  (\kappa \nabla p_F,\nabla q_F) & = (g,q_F),\qquad \forall q\in H_0^1(\Omega),
\end{align*}
\end{subequations}
where $\lambda \ge 0$ is a scaled Lam\'e coefficient, 
$c_0\ge 0$ is the storage coefficient, $\kappa$ is the (scaled) hydraulic conductivity,
and $\alpha$ is the (scaled) Biot-Willis coefficient.

The bilinear forms defining $\mathcal{A}((\cdot ;\cdot),(\cdot ;\cdot))$ are given by
\begin{align*}
a(\bu,\bv)&:= (\varepsilon( \bu), \varepsilon(\bv))+\lambda({\rm div} \bu, {\rm div}\bv),\qquad \forall \bu,\bv\in \bV,\\
b(\bv,q_F)&:= -\alpha({\rm div}\bv,q_F),\qquad \forall\bv\in \bV,\forall q_F\in Q, \\
c(p_F,q_F)&:= { c_0}(p_F,q_F)
+ (\kappa \nabla p_F,\nabla q_F),\qquad \forall p_F,q_F\in Q,
\end{align*}
where
$Q:=H_0^1(\Omega)$, $\bV:=\bH_0^1(\Omega)$. We define 
$\vert \cdot \vert_Q$, 
$\vert \cdot \vert_V$ to be 
\begin{align*}
\vert q_F\vert_Q^2 &:= 
\eta(q_F,q_F),
\qquad \forall q_F\in Q, \\ 
\vert \bv \vert_V^2 &:=(\varepsilon( \bv), \varepsilon(\bv))+\lambda({\rm div} \bv, {\rm div}\bv), \qquad \forall \bv \in \bV,
\end{align*}
where the parameter $\eta> 0$ is to be determined later.
As before, $a(\bv,\bv)\ge \vert \bv\vert_V^2$ for all $\bv\in \bV$, that is,  
~\eqref{coerc_a} is satisfied with $\underline{C}_a=1$. 
It remains to show~\eqref{inf_sup_b}. 
As in Example~\ref{example2}, it is easy to see that 
\begin{equation*}
\langle B\bv,\bar{Q}^{-1}B\bv\rangle_{Q'\times Q} \le \frac{\alpha^2}{\eta} ({\rm div}\bv, {\rm div}\bv),
\end{equation*}
where $B: \bV\rightarrow Q'$, $B:=-\alpha {\rm div}$. Therefore, we obtain
\begin{equation}\label{bound_v_norm_ex4}
\begin{array}{rl}
\Vert \bv\Vert_V^2 & = (\varepsilon( \bv), \varepsilon(\bv))+\lambda({\rm div} \bv, {\rm div}\bv) + \langle B\bv,\bar{Q}^{-1}B\bv\rangle_{Q'\times Q} \\
&\le  (\varepsilon( \bv), \varepsilon(\bv)) + \left(\lambda+\frac{\alpha^2}{\eta}\right)({\rm div}\bv, {\rm div}\bv)
\le \left(1+\lambda+\frac{\alpha^2}{\eta}\right) \Vert \bv\Vert_1^2.
\end{array}
\end{equation}
We choose $\bv_0$ such that $-{\rm div} \bv_0 = \frac{1}{\sqrt{1+\lambda}}q_F$ 
and use~\eqref{beta_s} to obtain $\Vert \bv_0\Vert_1 \le \frac{1}{\beta_s}\frac{1}{\sqrt{1+\lambda}} \Vert q_F\Vert$, and finally  

\begin{equation}\label{ex3}
\begin{array}{rl}
\sup_{\bv \in \bV} \frac{b(\bv,q_F)}{\Vert \bv\Vert_V} &\ge 
\frac{b(\bv_0,q_F)}{\Vert \bv_0\Vert_V} = \frac{\frac{\alpha}{\sqrt{1+\lambda}}\Vert q_F\Vert^2}{\Vert \bv_0\Vert_V} 
\ge \frac{\frac{\alpha}{\sqrt{1+\lambda}}}{\sqrt{\left(1+\lambda+\frac{\alpha^2}{\eta}\right)}}\frac{\Vert q_F\Vert^2}{\Vert \bv_0\Vert_1}
\\ 
&\ge \frac{\beta_s\alpha}{\sqrt{\left(1+\lambda+\frac{\alpha^2}{\eta}\right)}} \frac{\Vert q_F\Vert^2}{ \Vert q_F\Vert}
= \frac{\beta_s\alpha}{\sqrt{\left(1+\lambda+\frac{\alpha^2}{\eta}\right)}} \frac{1}{\sqrt{\eta}} \vert q_F\vert_Q.
\end{array}
\end{equation}
For $\eta:= \frac{\alpha^2}{(1+\lambda)}>0$ the right-hand side of~\eqref{ex3} is bounded from below by  
$\frac{\beta_s}{\sqrt{2}}\vert q_F\vert_Q $ which shows~\eqref{inf_sup_b} with $\underline{\beta}=\frac{1}{\sqrt{2}}\beta_s$. 
Note that there are also other possible choices for $\eta$.

We conclude that 
\begin{align*}
\mathcal{B}& :=
\begin{bmatrix}
(-{\rm div} \varepsilon-(1+\lambda) \nabla {\rm div})^{-1} & \\ 
& \left(\left(c_0+{\alpha^2}/{(1+\lambda)}\right)I -  {\rm div}\kappa \nabla\right)^{-1} 
\end{bmatrix} 
\end{align*}
provides a norm-equivalent preconditioner for the combined norm, where 
we have used~\eqref{bound_v_norm_ex4}. 
To solve the elasticity subproblem, one can use the multigrid method 
proposed in \cite{hong2016robust,hong2016uniformly}.
\end{example}

\begin{example}\label{example5}
By introducing $p_S=-\lambda{\rm div} \bu$ and  
substituting $\alpha p_F\rightarrow p_F, c_0\alpha^{-2}\rightarrow c_0, \kappa \alpha^{-2} \rightarrow \kappa, \alpha^{-1}g\rightarrow g$
in {\it Example} \ref{example4} we obtain the following three-field variational formulation 
of Biot's model, see~\cite{lee2017parameter},
\begin{subequations}
\begin{align*}
(\varepsilon( \bu), \varepsilon(\bv))-(p_S+p_F, {\rm div}\bv)& = (\bbf,\bv),  \qquad \forall \bv\in \bH_0^1(\Omega), \\
-({\rm div} \bu,q_S)-\lambda^{-1}(p_S,q_S) &= 0, \qquad \forall q_S\in L_0^2(\Omega),\\
-({\rm div}\bu,q_F)-{ c_0}(p_F,q_F)
-  (\kappa \nabla p_F,\nabla q_F) & = (g,q_F),\qquad \forall q_F\in H_0^1(\Omega).
\end{align*}
\end{subequations}
The bilinear forms that determine $\mathcal{A}((\cdot ;\cdot),(\cdot ;\cdot))$ are 
\begin{align*}
a(\bu,\bv)&:= (\varepsilon( \bu), \varepsilon(\bv)),\qquad \forall \bu,\bv\in \bV,\\
b(\bv,\bq)&:= -({\rm div}\bv,q_S)-({\rm div}\bv,q_F),\qquad \forall\bv\in \bV,\forall \bq\in \bQ, \\
c(\bp,\bq)&:= \lambda^{-1}(p_S,q_S)
+ { c_0}(p_F,q_F)+ (\kappa \nabla p_F,\nabla q_F),\qquad \forall \bp,\bq \in \bQ,
\end{align*}
where $\bV=\bH_0^1(\Omega)$, $\bQ=L_0^2(\Omega)\times H_0^1(\Omega)$ and 
$\bp = (p_S;p_F)$, $\bq=(q_S;q_F)$.
Then the operator $B$ is given by 
$$
B := \begin{pmatrix} -{\rm div} \\ -{\rm div}  \end{pmatrix}.
$$
Moreover, we define  
$\vert \cdot \vert_Q$, 
$\vert \cdot \vert_V$ to be 
\begin{align*}
\vert\bq\vert_Q^2 &:= 
\left(
\begin{pmatrix}
I & I \\
I & I
\end{pmatrix}
\begin{pmatrix} q_S\\ { q_{F,0}}  \end{pmatrix},
\begin{pmatrix} q_S\\ { q_{F,0}} \end{pmatrix}
\right)= \Vert q_S+{ q_{F,0}} \Vert^2,
\qquad \forall \bq\in \bQ, \\ 
\vert \bv \vert_V^2 &:=(\varepsilon( \bv), \varepsilon(\bv)), \qquad \forall \bv \in \bV,
\end{align*}
where $q_{F,0}:=P_0q_F$ and $P_0$ 
is the $L^2$ projection from $L^2(\Omega)$ to $L^2_0(\Omega)$. 
Then 
\begin{align*}
\Vert \bq\Vert_Q^2 & = 
\left(
\begin{pmatrix}
I & I \\
I & I
\end{pmatrix}
\begin{pmatrix} q_S\\ { q_{F,0}} \end{pmatrix},
\begin{pmatrix} q_S\\ { q_{F,0}} \end{pmatrix}
\right)
+ \left(\begin{pmatrix}
\lambda^{-1} I &0 \\ 0& { c_0}I - {\rm div} \kappa \nabla
 \end{pmatrix} 
\begin{pmatrix} q_S\\ q_F \end{pmatrix},
\begin{pmatrix} q_S\\ q_F \end{pmatrix} 
 \right)\\
 & 
=  \left(\begin{pmatrix}
(1+\lambda^{-1}) I &P_0 \\ P_0& P_0+{c_0}I - {\rm div} \kappa \nabla
 \end{pmatrix} 
\begin{pmatrix} q_S\\ q_F \end{pmatrix},
\begin{pmatrix} q_S\\ q_F \end{pmatrix} 
 \right)=(\bar{Q}\bq,\bq)
 =\langle \bar{Q}\bq,\bq\rangle_{\bQ'\times\bQ}.
\end{align*}
As in the previous examples, \eqref{coerc_a} is satisfied with $\underline{C}_a=1$.
Next, we choose $\bv_0$ such that $-{\rm div}\bv_0=q_S+{ q_{F,0}} $ for which we have $\Vert \bv_0\Vert_1 \le \beta_s^{-1}
\Vert q_S+{ q_{F,0}} \Vert$, see~\eqref{beta_s}. Then 
\begin{equation*}
b(\bv_0,\bq)= \Vert q_S+{ q_{F,0}} \Vert^2=\vert \bq\vert_Q^2.
\end{equation*}
Moreover, 
\begin{align*}
\Vert \bv_0\Vert_V^2 & =  (\varepsilon( \bv_0), \varepsilon(\bv_0))
+ (\bar{Q}^{-1}B\bv_0,B\bv_0) = (\varepsilon( \bv_0), \varepsilon(\bv_0)) 
+ \left(\bar{Q}^{-1} \begin{pmatrix}-{\rm div}\bv_0 \\-{\rm div}\bv_0  \end{pmatrix}, 
\begin{pmatrix} -{\rm div}\bv_0\\-{\rm div}\bv_0 \end{pmatrix}
 \right)\\ 
 & \le \Vert \bv_0 \Vert_1^2 
+ \frac14 \left(\bar{Q}^{-1} \begin{pmatrix} I & P_0 \\ P_0 & P_0 \end{pmatrix} \begin{pmatrix} {\rm div}\bv_0\\ {\rm div}\bv_0  \end{pmatrix}, 
 \begin{pmatrix} I & P_0 \\ P_0 & P_0 \end{pmatrix} \begin{pmatrix}  {\rm div}\bv_0\\ {\rm div}\bv_0  \end{pmatrix}
 \right) \\
 & 
 \le \Vert \bv_0 \Vert_1^2 
+ \frac14\left(\begin{pmatrix} I & P_0 \\ P_0& P_0 \end{pmatrix} \begin{pmatrix}  {\rm div}\bv_0\\ {\rm div}\bv_0 \end{pmatrix}, 
 \begin{pmatrix}  {\rm div}\bv_0\\ {\rm div}\bv_0  \end{pmatrix}
 \right) = \Vert \bv_0 \Vert_1^2 
+ ({\rm div}\bv_0,{\rm div}\bv_0)\\
 &\le \beta_s^{-2}\Vert q_S+{ q_{F,0}} \Vert^2 + \Vert q_S+{ q_{F,0}} \Vert^2 
 = (\beta_s^{-2}+1) \vert \bq \vert_Q^2.
\end{align*}
Now~\eqref{inf_sup_b} follows directly from    
\begin{align*}
\sup_{\bv \in V} \frac{b(\bv,\bq)}{\Vert \bv\Vert_V} &\ge 
\frac{b(\bv_0,\bq)}{\Vert \bv_0\Vert_V} \ge \frac{\vert \bq\vert_Q^2}{\sqrt{(\beta_s^{-2}+1)}
\vert \bq\vert_Q}=:\underline{\beta}\vert \bq\vert_Q, \qquad \forall \bq \in \bQ.
\end{align*}
Using the fitted norms for the constructions of a norm-equivalent preconditioner, cf.~\cite{mardal2011preconditioning}, results in
\begin{align*}
\mathcal{B}& :=
\begin{bmatrix}
\left(-{\rm div} \varepsilon\right)^{-1} & \\ 
& \begin{pmatrix}
(1+\lambda^{-1}) I &P_0 \\ P_0& P_0+{ c_0}I - {\rm div} \kappa \nabla
 \end{pmatrix}^{-1} 
\end{bmatrix}.
\end{align*}
\end{example}

\begin{remark}
In~\cite{lee2017parameter}, the authors showed that the three-field formulation for Biot's model  in {\it Example}~\ref{example5} 
is not stable under the $Q$-seminorm defined by $\vert \bq\vert_Q^2=\Vert p_S\Vert^2+\Vert p_F\Vert^2$.
\end{remark}

\begin{example}\label{example6}
By introducing the total pressure $p_T=p_S+p_F$ in {\it Example} \ref{example5}, another discrete in time
three-field formulation of the quasi-static Biot's consolidation model, see~\cite{lee2017parameter}, is obtained
and has the form
\begin{subequations} 
\begin{align*}
(\varepsilon( \bu), \varepsilon(\bv))-(p_T, {\rm div}\bv)& = (\bbf,\bv),  \qquad \forall \bv\in \bH_0^1(\Omega), \\
-({\rm div} \bu,q_T)-(\lambda^{-1}p_T,q_T)
+(\alpha \lambda^{-1} p_F,q_T) &= 0, \qquad \forall q_T\in L^2(\Omega),\\
(\alpha \lambda^{-1} p_T,q_F) - ({ (\alpha^2 \lambda^{-1}+c_0)}p_F,q_F)
-  (\kappa \nabla p_F,\nabla q_F) & = (g,q_F),\qquad \forall q_F\in H_0^1(\Omega). 
\end{align*} 
\end{subequations}
Here, $\mathcal{A}((\cdot ;\cdot),(\cdot ;\cdot))$ is constructed from  
\begin{align*}
a(\bu,\bv):=& (\varepsilon( \bu), \varepsilon(\bv)),\qquad \forall \bu,\bv\in \bV,\\
b(\bv,\bq):=& -({\rm div}\bv,q_T),\qquad \forall\bv\in \bV,\forall \bq\in \bQ, \\
c(\bp,\bq):=&(\lambda^{-1}p_T,q_T)
-(\alpha \lambda^{-1} p_F,q_T) - (\alpha \lambda^{-1} p_T,q_F) + ({ (\alpha^2 \lambda^{-1}+c_0)}p_F,q_F) 
\\ &
+  (\kappa \nabla p_F,\nabla q_F),\qquad \forall \bp,\bq \in \bQ,
\end{align*}
where $\bV=\bH_0^1(\Omega)$, $\bQ=L^2(\Omega)\times H_0^1(\Omega)$ and 
$\bp = (p_T;p_F)$, $\bq=(q_T;q_F)$. Obviously, the operator $B: \bV \rightarrow \bQ'$ is defined by
\begin{align*}
B:=\begin{pmatrix}-{\rm div}\\0 \end{pmatrix}.
\end{align*}
We next set  
\begin{align*}
\vert\bq\vert_Q^2 &:= (q_{T,0},q_{T,0}), \qquad \forall \bq\in \bQ, \\ 
\vert \bv \vert_V^2 &:=(\varepsilon( \bv), \varepsilon(\bv)), \qquad \forall \bv \in \bV,
\end{align*}
where $q_{T,0}:=P_0 q_T$ is the $L^2$ projection of $q_T\in L^2(\Omega)$ to $L_0^2(\Omega)$.
Now we choose $\bv_0\in \bV= \bH_0^1(\Omega)$ such that 
$-{\rm div}\bv_0= q_{T,0}$ for which it holds 
$\Vert \bv_0\Vert_1 \le \beta_s^{-1}\Vert q_{T,0}\Vert=\beta_s^{-1} \vert \bq\vert_Q$, 
see~\eqref{beta_s}, and also 
\begin{align*}
b(\bv_0,\bq)=(q_{T,0},q_{T,0})=\vert \bq\vert_Q^2.
\end{align*}
From the definition of $\Vert \cdot\Vert_Q$, and using similar arguments as in the previous examples, we obtain  
\begin{align*}
\Vert \bv_0\Vert_V^2 &= (\varepsilon( \bv_0), \varepsilon(\bv_0)) + 
\langle B \bv_0, \bar{Q}^{-1} B\bv_0\rangle \le  (\varepsilon( \bv_0), \varepsilon(\bv_0))  
+ ({\rm \div} \bv_0, {\rm div}\bv_0) \le   2\Vert \bv_0\Vert_1^2 \le 2\beta_s^{-2}\vert \bq\vert_Q^2.
\end{align*}
Again, \eqref{coerc_a} is satisfied with $\underline{C}_a=1$ while~\eqref{inf_sup_b} follows from 
\begin{align*}
\sup_{\bv \in V} \frac{b(\bv,\bq)}{\Vert \bv\Vert_V} &\ge 
\frac{b(\bv_0,\bq)}{\Vert \bv_0\Vert_V} \ge \beta_s 
\frac{\vert \bq\vert_Q^2}{
\vert \bq\vert_Q}=:\underline{\beta}\vert \bq\vert_Q, \qquad \forall \bq \in \bQ.
\end{align*}

Thus, the fitted norms generate the norm-equivalent preconditioner 
\begin{align*}
\mathcal{B}& :=
\begin{bmatrix}
\left(-{\rm div} \varepsilon\right)^{-1} & \\ 
& \begin{pmatrix}
 \lambda^{-1} I +P_0  &-\alpha \lambda^{-1} I  \\ -\alpha \lambda^{-1} I & \alpha^2 \lambda^{-1} I +{ c_0}I - {\rm div} \kappa \nabla
 \end{pmatrix}^{-1} 
\end{bmatrix}.
\end{align*}
\begin{remark}
The arguments presented above are valid for a vanishing storage coefficient, i.e., $c_0=0$.
Moreover, our analysis shows how {\it Example}~\ref{example5} and {\it Example}~\ref{example6} are related to each other. 
In fact, by the transformation $\begin{pmatrix} p_T \\ p_F \end{pmatrix}= \begin{pmatrix} I & I \\ 0 & I \end{pmatrix} \begin{pmatrix} p_S \\ p_F \end{pmatrix}$ or,  equivalently, 
$\begin{pmatrix} p_S \\ p_F \end{pmatrix}= \begin{pmatrix} I & -I \\ 0 & I \end{pmatrix} \begin{pmatrix} p_T \\ p_F \end{pmatrix}$, we can derive the stability and preconditioners of {\it Example}~\ref{example5} and {\it Example}~\ref{example6}  from each other. 
\end{remark}

\begin{remark}
Note that for the case $c_0=\alpha^2 \lambda^{-1}$, as considered in \cite{lee2017parameter}, we can further estimate
\begin{align*}
c(\bq,\bq):=&(\lambda^{-1}q_T,q_T)
-2(\alpha \lambda^{-1} q_F,q_T) + 2(\alpha^2 \lambda^{-1}q_F,q_F) 
+  (\kappa \nabla q_F,\nabla q_F)\\
\ge& \frac{1}{4}(\lambda^{-1}q_T,q_T)+ \frac23(\alpha^2 \lambda^{-1}q_F,q_F) 
+  (\kappa \nabla q_F,\nabla q_F)\\
\ge& \frac{1}{4}\left((\lambda^{-1}q_T,q_T)+(\alpha^2 \lambda^{-1}q_F,q_F) 
+  (\kappa \nabla q_F,\nabla q_F)\right).
\end{align*}
Hence, we obtain
\begin{align*}
\Vert \bq\Vert^2_Q=\vert \bq\vert^2_Q+c(\bq,\bq)\ge \frac{1}{4}\left((q_{T,0},q_{T,0})+ (\lambda^{-1}q_T,q_T)+(\alpha^2 \lambda^{-1}q_F,q_F) 
+  (\kappa \nabla q_F,\nabla q_F)\right), 
\end{align*}
from which we conclude the stability result and preconditioner shown in~\cite{lee2017parameter}: 
\begin{align*}
\mathcal{B}_0& :=
\begin{bmatrix}
\left(-{\rm div} \varepsilon\right)^{-1} & \\ 
& \begin{pmatrix}
 \lambda^{-1} I +P_0  &  \\ & \alpha^2 \lambda^{-1} I - {\rm div} \kappa \nabla
 \end{pmatrix}^{-1} 
\end{bmatrix}.
\end{align*}
\end{remark}
\end{example}

\begin{example}\label{example7}
Next we consider a four-field formulation of Biot's model in which a total pressure has been introduced~\cite{Kumar2020conservative}.
The variational problem reads as 
\begin{subequations}
\begin{align*}
2\mu(\varepsilon( \bu), \varepsilon(\bv))+(p_T,{\rm div}\bv)& = (\bbf,\bv),  \qquad \forall \bv\in \bH_0^1(\Omega), \\
\frac{1}{\tau \kappa}(\bw,\bz) -(p,{\rm div}\bz) &= 0,\qquad\forall \bz \in \bH_0({\rm div},\Omega), \\
({\rm div}\bu,q_T)-\lambda^{-1}(p_T,q_T)- \frac{\alpha}{\lambda}(p,q_T) 
 &= 0, \qquad \forall q_T \in L^2_0(\Omega),
\\
-({\rm div}\bw,q)-\frac{\alpha}{\lambda}(p_T,q)-  \left({ c_0}+\frac{\alpha^2}{\lambda}\right)(p,q)
 & = (f,q), \qquad \forall q\in L^2_0(\Omega),
\end{align*} 
\end{subequations}
with $\mu$ and $\lambda$ the Lam\'e coefficients, 
$\tau$ the time step size, $\kappa$ the hydraulic conductivity, $\alpha$ the 
Biot-Willis coefficient, and $c_0\ge 0$ the constrained specific storage coefficient.

The bilinear forms defining $\mathcal{A}((\cdot ;\cdot),(\cdot ;\cdot))$ in this example are given by 
\begin{align*}
a(\bbu,\bbv):=& 2\mu (\varepsilon( \bu), \varepsilon(\bv))
+\frac{1}{\tau \kappa}(\bw,\bz)
,\qquad \forall \bbu,\bbv\in \bV,\\
b(\bbv,\bq):=& ({\rm div}\bv,q_T)-({\rm div} \bz,q),\qquad \forall\bbv\in \bV,\forall \bq\in \bQ, \\
c(\bp,\bq):=&\lambda^{-1}(p_T,q_T)
+ \frac{\alpha}{\lambda}(p,q_T) +\frac{\alpha}{\lambda}(p_T,q)
+   \left({ c_0}+\frac{\alpha^2}{\lambda}\right)(p,q),\qquad \forall \bp,\bq \in \bQ,
\end{align*}
where 
$\bbu:=(\bu;\bw), \bbv:=(\bv;\bz) \in \bV =\bH_0^1(\Omega)\times \bH_0({\rm div},\Omega),$
$\bp:=(p_T;p), \bq:=(q_T;q) \in \bQ=L^2_0(\Omega)\times L^2_0(\Omega).$
Hence, $B$ has the form
\begin{align*}
B=\begin{pmatrix} 
{\rm div} &0 \\ 0& -{\rm div}
\end{pmatrix}.
\end{align*}
Using {the norm fitting technique, we define} 
\begin{align*}
\vert\bq\vert_Q^2 &:= \frac{1}{2\mu}
(q_{T},q_{T})+\tau \kappa (q,q), \qquad \forall \bq\in \bQ, \\ 
\vert \bbv \vert_V^2 &:=2\mu(\varepsilon( \bv), \varepsilon(\bv))+
\frac{1}{\tau \kappa}(\bz,\bz), \qquad \forall \bbv \in \bV.
\end{align*} 
For this choice of $\vert \cdot\vert_V$ the coercivity estimate~\eqref{coerc_a} is 
again fulfilled with $\underline{C}_a=1$.

Now, in view of the Stokes and Darcy inf-sup 
conditions~\eqref{beta_s} and \eqref{beta_d}, for any $\bq=(q_T; q)$ 
we can choose $\bbv_0=(\bv_0;\bz_0)$ such that ${\rm div}\bv_0=\frac{1}{2\mu} q_T$ and 
$-{\rm div}\bz_0=\tau \kappa\,q$ and it holds that

\begin{align}\label{divStokes}
\Vert \bv_0\Vert_1 \le\frac{\beta_s^{-1}}{2\mu} \Vert q_T\Vert, \qquad 
\Vert \bz_0\Vert_{\rm div} \le \beta_d^{-1}  \tau \kappa \Vert q\Vert.
\end{align}
Then we have 
\begin{align}\label{bv0q}
b(\bbv_0, \bq)=& ({\rm div}\bv_0,q_T)-({\rm div} \bz_0,q)= \frac{1}{2\mu}(q_T,q_T)+\tau\kappa (q,q)=\vert\bq\vert_Q^2  
\end{align}
and 
\begin{align*}
\Vert \bbv_0\Vert_V^2&=2\mu(\varepsilon( \bv_0), \varepsilon(\bv_0))+\frac{1}{\tau \kappa}(\bz_0,\bz_0)+\langle B \bbv_0, \bar{Q}^{-1} B\bbv_0\rangle\\
&= 2\mu(\varepsilon( \bv_0), \varepsilon(\bv_0))+\frac{1}{\tau \kappa}(\bz_0,\bz_0)+ \left(\bar{Q}^{-1}  \begin{pmatrix} {\rm div}\bv_0 \\ -{\rm div}\bz_0\end{pmatrix}, 
 \begin{pmatrix} {\rm div}\bv_0\\ -{\rm div}\bz_0 \end{pmatrix}  \right) \\
 &\le 2\mu \Vert\bv_0\Vert_1^2+\frac{1}{\tau \kappa}\Vert\bz_0\Vert^2+ \left( \begin{pmatrix} 2\mu I & 0 \\ 0 & \frac{1}{\tau\kappa}I\end{pmatrix}\begin{pmatrix} {\rm div}\bv_0 \\ -{\rm div}\bz_0\end{pmatrix}, 
 \begin{pmatrix} {\rm div}\bv_0\\ -{\rm div}\bz_0 \end{pmatrix} \right)\\
 &= 2\mu \Vert\bv_0\Vert_1^2+\frac{1}{\tau \kappa}\Vert\bz_0\Vert^2+2\mu \Vert {\rm div}\bv_0\Vert^2+\frac{1}{\tau\kappa}\Vert {\rm div} \bz_0\Vert^2\\
 &\le 4\mu \Vert\bv_0\Vert_1^2+\frac{1}{\tau \kappa}\Vert\bz_0\Vert_{{\rm div}}^2.
  \end{align*}
 Now by \eqref{divStokes} and the definition of $\vert \cdot\vert_Q$, we obtain  
\begin{equation}
\label{estv0}
\begin{array}{rl}
\Vert \bbv_0\Vert_V^2&\le 4\mu \Vert\bv_0\Vert_1^2+\frac{1}{\tau \kappa}\Vert\bz_0\Vert_{{\rm div}}^2\le 4\mu \frac{\beta_s^{-2}}{4\mu^2} \Vert q_T\Vert^2+\frac{1}{\tau \kappa} \beta_d^{-2}  \tau^2 \kappa^2 \Vert q\Vert^2\\
&\le 2\max\{\beta_s^{-2}, \beta_d^{-2}\}\left( \frac{1}{2\mu} \Vert q_T\Vert^2+\tau \kappa \Vert q\Vert^2 \right)\le 2\max\{\beta_s^{-2}, \beta_d^{-2}\} \vert \bq\vert^2_Q.
 \end{array}
 \end{equation}
Hence, in this example \eqref{inf_sup_b} follows from 
\begin{align*}
\sup_{\bbv \in \bV} \frac{b(\bbv,\bq)}{\Vert \bbv\Vert_V} &\ge 
\frac{b(\bbv_0,\bq)}{\Vert \bbv_0\Vert_V}\ge \frac{\vert \bq\vert_Q^2}{\sqrt{2\max\{\beta_s^{-2}, \beta_d^{-2}\} }
\vert \bq\vert_Q}=:\underline{\beta}\vert \bq\vert_Q, \qquad \forall \bq\in \bQ,
\end{align*}
where we have used \eqref{bv0q} and \eqref{estv0}. 

From our findings we conclude that 
\begin{align*}
\mathcal{B} & : = 
\begin{bmatrix}
\left(
\begin{bmatrix}
-2\mu {\rm div}\varepsilon & \\
&(\tau \kappa)^{-1}I
\end{bmatrix}
 + 
 \begin{bmatrix}
-\nabla & \\
& \nabla
\end{bmatrix} 
\mathcal{C}^{-1}
\begin{bmatrix}
\rm{div} & \\ 
& -\rm{div}
\end{bmatrix}\right)^{-1}
\\
& \mathcal{C}^{-1}
\end{bmatrix}\\
& \eqsim 
\begin{bmatrix}
(-2\mu {\rm div}\varepsilon)^{-1} & & \\
&\tau\kappa (I -\nabla {\rm div})^{-1} &
\\
& & \mathcal{C}^{-1}
\end{bmatrix}
,
\end{align*}
provides a norm-equivalent preconditioner, 
where 
\begin{align*}
\mathcal{C}:=
\begin{bmatrix}
(\lambda^{-1} +(2\mu)^{-1})I & \alpha \lambda^{-1}I \\
\alpha \lambda^{-1}I & (c_0+\alpha^2\lambda^{-1}+\tau \kappa)I
\end{bmatrix} \quad
\text{and} \quad
\mathcal{C}^{-1}=
\eta\begin{bmatrix}
2 \mu (\alpha^2 + \lambda(c_0+\tau \kappa)) I & -2\mu \alpha I \\
-2 \mu \alpha I & (\lambda + 2\mu) I
\end{bmatrix}
,
\end{align*}
with $\eta=1/(\alpha^2 +(\lambda+2\mu)(c_0+\tau\kappa))$.
\end{example}

\begin{example}\label{example8}
Finally, let us consider the classical three-field formulation of Biot's consolidation model as analyzed in~\cite{HongKraus2018}.
After some rescaling of parameters and semi-discretization in time by the implicit Euler method the static variational problem
to be
solved in each time step is given by
\begin{subequations}
\begin{align*}
(\varepsilon( \bu), \varepsilon(\bv)) + \lambda_\mu ({\rm div}\bu,{\rm div}\bv)-(p, {\rm div}\bv)& = (\bbf,\bv),  \qquad \forall \bv\in \bH_0^1(\Omega), \\
R_p^{-1} (\bw,\bz) -(p,{\rm div} \bz) &= 0, \qquad \forall \bz \in \bH_0({\rm div},\Omega) \\
-({\rm div} \bu,q) -({\rm div} \bw,q) -\alpha_p (p,q) &= (g,q), \qquad \forall q \in L^2_0(\Omega),
\end{align*} 
\end{subequations}
with parameters $\lambda_\mu=\lambda/(2 \mu)$, $R_p^{-1}=\alpha^2 \tau^{-1} \kappa^{-1}$, and $\alpha_p=c_0 \alpha^{-2}$.
In this example, we set $\bV:=\bH_0^1(\Omega) \times \bH_0({\rm div},\Omega)$, $Q:=L^2_0(\Omega)$. 
The individual
bilinear forms 
are defined by
\begin{align*}
a(\bar{\bu},\bar{\bv}) &= a((\bu;\bw),(\bv;\bz)) := (\varepsilon( \bu), \varepsilon(\bv)) + \lambda_\mu ({\rm div}\bu,{\rm div}\bv) + R_p^{-1} (\bw,\bz),
\quad\forall \bar{\bu}, \bar{\bv} \in \bV,\\
b(\bar{\bv},q)&=b((\bv;\bz),q):= -({\rm div}\bv + {\rm div}\bz,q) =: -(\Divv\bar{\bv},q), \quad \forall \bar{\bv} \in \bV, \forall q \in Q, \\
c(p,q)&:= \alpha_p (p,q),\quad \forall p,q \in Q,
\end{align*}
and $B: \bV \rightarrow Q'$, $B=-\Divv$ where $\Divv \bar{\bv} = \Divv (\bv;\bz) := {\rm div}\bv + {\rm div}\bz$.

The fitted norms we use in this example are determined by 
\begin{align*}
\vert q \vert_Q^2 &:= \left( R_p + \frac{1}{1+\lambda_\mu}\right) (q,q), \qquad \forall q\in Q, \\ 
\vert \bar{\bv} \vert_V^2 &:=a(\bar{\bv},\bar{\bv}), \qquad \forall \bar{\bv} \in \bV.
\end{align*}
The coercivity estimate~\eqref{coerc_a} is 
again trivially fulfilled with $\underline{C}_a=1$.

Next, from the inf-sup conditions \eqref{beta_d} and \eqref{beta_s} we infer that for any $q \in Q$ there exist $\bv_0$ and $\bz_0$
such that $\bar{\bv}_0:=(\bv_0;\bz_0) \in \bV$ such that 
\begin{subequations}\label{choice}
\begin{eqnarray}
-{\rm div} \bz_0 &=& R_p q \quad \mbox{ and } \quad \Vert \bz_0 \Vert^2_{\rm div} \le \beta_d^{-2} R_p^2 \Vert q \Vert^2, \label{choice_d} \\
-{\rm div} \bv_0 &=& \frac{1}{1+\lambda_\mu} q \quad \mbox{ and } \quad \Vert \bv_0 \Vert^2_1 \le \beta_s^{-2} \frac{1}{(1+\lambda_\mu)^2} \Vert q \Vert^2. \label{choice_s}
\end{eqnarray}
\end{subequations}
Now, from~\eqref{choice_d} and \eqref{choice_s} we conclude
\begin{equation}\label{est_d}
\begin{aligned}
\beta_d^{-2} R_p \Vert q \Vert^2 &\ge R_p^{-1} \Vert \bz_0 \Vert^2_{\rm div} = R_p^{-1} \left( \Vert {\rm div}\bz_0 \Vert^2 + \Vert \bz_0 \Vert^2 \right)  = \Vert R_p^{-1/2} {\rm div}\bz_0 \Vert^2 + \Vert R_p^{-1/2} \bz_0 \Vert^2,
\end{aligned}
\end{equation}
\begin{equation}\label{est_s}
\begin{aligned}
\beta_s^{-2} \frac{1}{1+\lambda_\mu} \Vert q \Vert^2 &\ge (1+\lambda_\mu) \Vert \bv_0 \Vert^2_1
=\frac{1}{2} (1+\lambda_\mu) 2 \left( \Vert \nabla \bv_0 \Vert^2 + \Vert \bv_0 \Vert^2 \right) \\
& \ge \frac{1}{2} \left[
(\varepsilon(\bv_0),\varepsilon(\bv_0)) + \lambda_\mu ({\rm div} \bv_0,{\rm div} \bv_0) + (1+\lambda_\mu)  ({\rm div} \bv_0,{\rm div} \bv_0) \right]. 
\end{aligned}
\end{equation}
Choosing $\underline{\beta} := \min \{ \beta_d, \frac{\beta_s}{\sqrt{2}} \}$, we obtain
\begin{align}\label{infex7}
\nonumber \underline{\beta}^{-2} \vert q \vert_Q^2 &= \underline{\beta}^{-2} \left( R_p +\frac{1}{1+\lambda_\mu} \right) \Vert q \Vert^2 \\
& \nonumber 
 \ge  (\varepsilon(\bv_0),\varepsilon(\bv_0)) + \lambda_\mu ({\rm div} \bv_0,{\rm div} \bv_0) +  R_p^{-1} (\bz_0, \bz_0) \\
&\nonumber 
+ R_p^{-1} ({\rm div} \bz_0, {\rm div} \bz_0) + \left( \frac{1}{1+\lambda_\mu} \right)^{-1} ({\rm div} \bv_0, {\rm div} \bv_0) \\
&\ge\vert \bar{\bv}_0 \vert_V^2 + \frac{1}{2} (R_p+\frac{1}{1+\lambda_\mu}+\alpha_p)^{-1} (\Divv \bar{\bv}_0, \Divv \bar{\bv}_0){ \ge \frac12\Vert \bar{\bv}_0\Vert_V^2.}
\end{align}
Hence, in this example \eqref{inf_sup_b} follows 
by combining \eqref{choice_d}, \eqref{choice_s} and \eqref{infex7}. 

\begin{remark}
Noting that for any $\epsilon\in (0,1)$, by Cauchy inequality, we have
\begin{align*}
\Vert \bar{\bv}\Vert_V^2&=\vert \bar{\bv}\vert_V^2+ \langle B\bar{\bv}, \bar{Q}^{-1} B\bar{\bv}\rangle_{Q'\times Q} 
=a(\bar{\bv},\bar{\bv})+ (R_p+\frac{1}{1+\lambda_\mu}+\alpha_p)^{-1}  (\Divv \bar{\bv}, \Divv \bar{\bv})\\
&= (\varepsilon( \bv), \varepsilon(\bv)) + \lambda_\mu ({\rm div}\bv,{\rm div}\bv) + R_p^{-1} (\bz,\bz)\\
&+(R_p+\frac{1}{1+\lambda_\mu}+\alpha_p)^{-1}  ({\rm div} \bv+{\rm div} \bz, {\rm div} \bv+{\rm div} \bz)\\
&\ge (\varepsilon( \bv), \varepsilon(\bv)) + \lambda_\mu ({\rm div}\bv,{\rm div}\bv) + R_p^{-1} (\bz,\bz)\\
&+(\frac{1}{1+\lambda_\mu})^{-1} (1-\epsilon^{-1})({\rm div} \bv, {\rm div} \bv)+(R_p+\frac{1}{1+\lambda_\mu}+\alpha_p)^{-1}  (1-\epsilon)({\rm div} \bz, {\rm div} \bz)\\
&\ge \frac12 (\varepsilon( \bv), \varepsilon(\bv)) + \frac12\lambda_\mu ({\rm div}\bv,{\rm div}\bv) + R_p^{-1} (\bz,\bz)\\
&+({1+\lambda_\mu}) (\frac32-\epsilon^{-1})({\rm div} \bv, {\rm div} \bv)+(R_p+\frac{1}{1+\lambda_\mu}+\alpha_p)^{-1} (1-\epsilon)({\rm div} \bz, {\rm div} \bz). 
\end{align*}
Taking $\epsilon=\frac23$, we have 
\begin{align*}
\Vert \bar{\bv}\Vert_V^2&\ge \frac13 \left(\Vert\varepsilon(\bv)\Vert^2+ \lambda_\mu \Vert{\rm div}\bv\Vert^2+ R_p^{-1} \Vert\bz\Vert^2+(R_p+\frac{1}{1+\lambda_\mu}+\alpha_p)^{-1}\Vert{\rm div} \bz\Vert^2\right),
\end{align*}
from which we conclude the uniform stability result for classical three-field formulation of Biot's consolidation model presented in~\cite{HongKraus2018}.
\end{remark}

A norm-equivalent preconditioner in this final example is, therefore, given by
\begin{align*}
\mathcal{B}:=\begin{bmatrix}
-({\rm div}\varepsilon + \lambda_\mu \nabla {\rm div})^{-1} & &
\\
& (R_p^{-1} I - \nabla  (R_p+\frac{1}{1+\lambda_\mu}+\alpha_p)^{-1} {\rm div})^{-1} &\\
&  & ((R_p+\frac{1}{1+\lambda_\mu}+\alpha_p)I)^{-1} 
\end{bmatrix}.
\end{align*}

\end{example}

\section{Conclusion}

In this paper we have presented a new framework for the stability analysis of perturbed saddle-point problems in variational formulation
in a Hilbert space setting. Our aproach is constructive and is based on a specific norm fitting technique. The main theoretical
result (Theorem~\ref{our_theorem}) is a generalization of the classical splitting theorem (Theorem~\ref{brezzi1}) of Brezzi and
allows to conclude the necessary stability condition (big inf-sup condition) according to Babu\v{s}ka's theory (Theorem~\ref{thm:1})
from conditions similar to those on which Brezzi's theorem is based also in presence of a symmetric positive semidefinite perturbation
term $c(\cdot,\cdot) \not \equiv 0$.

As demonstrated on mixed formulations of generalized Poisson, Stokes, vector Laplacian, and Biot's equations, 
the new
norm fitting technique guids the process of defining proper parameter-dependent norms and allows for simple and short proofs of the
stability of perturbed saddle-point problems. 
Although the examples in the present paper are continuous (infinite-dimensional) models, the abstract framework suggests that the
proposed technique is also quite useful when studying the stability of various mixed (finite element) discretizations and
has a wide range of applications.

\bibliographystyle{plain}
\bibliography{stability}

\end{document}